\documentclass[3pt]{amsart}

\usepackage{amssymb,amscd,amsthm, verbatim,amsmath,color,fancyhdr, mathrsfs}
\usepackage{graphicx}
\usepackage{turnstile}
\usepackage{tikz}
\usepackage{multicol}
\usetikzlibrary{patterns}

\usepackage[letterpaper, left=2.5cm, right=2.5cm, top=2.5cm,
bottom=2.5cm,dvips]{geometry}

\newtheorem{thm}{Theorem}[section]

\newtheorem{lem}[thm]{Lemma}
\newtheorem{rem}[thm]{Remark}
\newtheorem{cor}[thm]{Corollary}

\newtheorem{defn}[thm]{Definition}

\newcommand{\vertiii}[1]{{\left\vert\kern-0.25ex\left\vert\kern-0.25ex\left\vert #1 
    \right\vert\kern-0.25ex\right\vert\kern-0.25ex\right\vert}}

\numberwithin{equation}{section}

\title[Semilinear parabolic systems with space-time forcing terms]{Blow-up and global existence for semilinear parabolic systems with space-time forcing terms}

\author[AZ. Fino, M. Jleli, B. Samet]{Ahmad Z. Fino, Mohamed Jleli, Bessem Samet}



\begin{document}

\maketitle

\begin{abstract}
We investigate the local existence,  finite time blow-up and  global existence of 
sign-changing solutions to the inhomogeneous parabolic system with space-time forcing terms
$$
u_t-\Delta u =|v|^{p}+t^\sigma w_1(x),\,\, v_t-\Delta v =|u|^{q}+t^\gamma w_2(x),\,\, (u(0,x),v(0,x))=(u_0(x),v_0(x)),
$$
where $t>0$, $x\in \mathbb{R}^N$, $N\geq 1$, $p,q>1$, $\sigma,\gamma>-1$, $\sigma,\gamma\neq0$, $w_1,w_2\not\equiv0$, and $u_0,v_0\in C_0(\mathbb{R}^N)$. For the finite time blow-up, two cases are discussed under the conditions $w_i\in L^1(\mathbb{R}^N)$ and $\int_{\mathbb{R}^N} w_i(x)\,dx>0$, $i=1,2$. Namely, if $\sigma>0$ or $\gamma>0$, we show that the (mild) solution $(u,v)$ to  the considered system blows up in finite time, while if  $\sigma,\gamma\in(-1,0)$, then a finite time blow-up occurs when $\frac{N}{2}< \max\left\{\frac{(\sigma+1)(pq-1)+p+1}{pq-1},\frac{(\gamma+1)(pq-1)+q+1}{pq-1}\right\}$. Moreover, if  $\frac{N}{2}\geq \max\left\{\frac{(\sigma+1)(pq-1)+p+1}{pq-1},\frac{(\gamma+1)(pq-1)+q+1}{pq-1}\right\}$, $p>\frac{\sigma}{\gamma}$ and $q>\frac{\gamma}{\sigma}$,
we show that the solution is global for  suitable initial values and $w_i$, $i=1,2$. 

 \end{abstract}

\section{Introduction}

This paper is concerned with the Cauchy problem for the inhomogeneous semilinear parabolic system with space-time forcing terms
\begin{eqnarray}\label{1}
\left\{
  \begin{array}{lll}
u_t-\Delta u &=&|v|^{p}+t^\sigma w_1(x),\quad x\in\mathbb{R}^N,t>0,\\
v_t-\Delta v &=&|u|^{q}+t^\gamma w_2(x),\quad x\in\mathbb{R}^N,t>0,
  \end{array}
\right.
\end{eqnarray}
supplemented with the initial conditions
\begin{equation}\label{Initialcondition1}
(u(0,x),v(0,x))=(u_0(x),v_0(x)),\quad x\in\mathbb{R}^N,
\end{equation}
where $u_0,v_0\in C_0(\mathbb{R}^N),$ $N\geq 1$, $p,q>1$, $\sigma,\gamma>-1$, $\sigma,\gamma\neq0$, and $w_1,w_2\not\equiv0$. 
Here, $C_0(\mathbb{R}^N)$ denotes the set of all continuous functions decaying to zero at infinity. Namely, we investigate the finite time blow-up and global existence of sign-changing solutions to  problem \eqref{1}--\eqref{Initialcondition1}.  We mention below some motivations for studying the considered problem. 

In the special case $w_1=w_2\equiv 0$ and $u,v\geq 0$,  \eqref{1} reduces to the homogeneous parabolic system
\begin{eqnarray}\label{HS}
\left\{
  \begin{array}{lll}
u_t-\Delta u &=&v^{p},\quad x\in\mathbb{R}^N,t>0,\\
v_t-\Delta v &=&u^{q},\quad x\in\mathbb{R}^N,t>0.
  \end{array}
\right.
\end{eqnarray}
Escobedo and Herrero \cite{EH} studied problem \eqref{HS}--\eqref{Initialcondition1}, where  $u_0,v_0\geq 0$,  bounded and continuous. It was shown that the critical exponent for this problem  is equal to $1+\frac{2}{N}(\max\{p,q\}+1)$, i.e.
\begin{itemize}
\item[(a)] if $1<pq\leq 1+\frac{2}{N}(\max\{p,q\}+1)$, then any nontrivial solution to \eqref{HS}--\eqref{Initialcondition1} blows up in finite time;
\item[(b)] if $pq>1+\frac{2}{N}(\max\{p,q\}+1)$, then \eqref{HS}--\eqref{Initialcondition1} admits global solutions for  small  initial values.
\end{itemize}
Observe that  in the case $p=q$ and $u_0=v_0$, \eqref{HS}--\eqref{Initialcondition1} reduces to  a scalar Cauchy problem, namely
\begin{eqnarray}\label{FP}
\left\{
  \begin{array}{lll}
u_t-\Delta u &=&u^{p},\quad x\in\mathbb{R}^N,t>0,\\
u(x,0) &=&u_0(x),\quad x\in\mathbb{R}^N.
  \end{array}
\right.
\end{eqnarray}
From Fujita \cite{Fujita} (see also \cite{AW,KS}), it is knwon that,
\begin{itemize}
\item[(a)] if $1<p\leq 1+\frac{2}{N}$ and $u_0\geq 0$, then any nontrivial solution to \eqref{FP}--\eqref{Initialcondition1} blows up in finite time;
\item[(b)] if $p>1+\frac{2}{N}$ and $u_0>0$  is smaller than a small Gaussian, then  \eqref{FP}--\eqref{Initialcondition1} admits global positive solutions.
\end{itemize}
These are precisely the conditions obtained in \cite{EH} under the assumption $p=q$. The number $1+\frac{2}{N}$  is said to be critical in the sense of Fujita.  

In the special case $\sigma=\gamma=0$,  \eqref{1} reduces to 

\begin{eqnarray}\label{1-0}
\left\{
  \begin{array}{lll}
u_t-\Delta u &=&|v|^{p}+ w_1(x),\quad x\in\mathbb{R}^N,t>0,\\
v_t-\Delta v &=&|u|^{q}+ w_2(x),\quad x\in\mathbb{R}^N,t>0. 
  \end{array}
\right.
\end{eqnarray}
Problem \eqref{1-0}--\eqref{Initialcondition1} was investigated by Bandle et al. \cite{BLZ}. To prove global nonexistence, it was  assumed that for $i=1,2$,  
\begin{equation}\label{assB}
\int_{\mathbb{R}^N} w_i(x)\,dx>0,\quad \int_{|x|>R} \frac{w_i^-(y)}{|x-y|^{N-2}}\,dy =\frac{o(1)}{|x|^{N-2}},\,\, R\gg 1,
\end{equation}
where $w_i^\pm=\max\{\pm w,0\}$. Namely, it was shown that,
\begin{itemize}
\item[(a)] if $p \geq q > 1$ and $\frac{p(q+1)}{pq-1}>\frac{N}{2}$,
problem \eqref{1-0}--\eqref{Initialcondition1} possesses no global solution,  when both $w_1$ and $w_2$ satisfy \eqref{assB};
\item[(b)] if $p\geq q > 1$ and $\frac{p(q+1)}{pq-1}=\frac{N}{2}$, problem \eqref{1-0}--\eqref{Initialcondition1} possesses no global solution, when  either of $w_1$ or $w_2$ satisfies \eqref{assB}  and $u_0,v_0 \geq  0$;
\item[(c)]  if $p\geq q > 0$  and $\frac{p(q+1)}{pq-1}<\frac{N}{2}$, then problem \eqref{1-0}--\eqref{Initialcondition1} has global positive solutions whenever $w_1(x), w_2(x), u_0(x), v_0(x)$ are  all nonnegative and are bounded above by $\frac{\epsilon}{\left(1+|x|^{N+\tau}\right)}$ for some
$\tau>0$ and some sufficiently small $\epsilon > 0$.
\end{itemize}
Note that system \eqref{1-0} (with positive solutions) was  invetigated by Zhang \cite{ZhangV} in a non-compact complete Riemannian manifold. 
For other contributions related to inhomogeneous problems, see, for example \cite{CY,Y,Zeng} and the references therein.

Very recently, Jleli et al. \cite{JKS} studied the scalar  case of  problem \eqref{1}, namely 
\begin{eqnarray}\label{eqJKS}
\left\{\begin{array}{lll}
u_t-\Delta u &=&|u|^{p}+t^\sigma w(x), \quad x\in\mathbb{R}^N, t>0,\\
u(x,0)&=&u_0(x),\quad x\in \mathbb{R}^N,
\end{array}
\right.
\end{eqnarray}
where  $N\geq 2$, $p>1$, $\sigma>-1$, $\sigma\neq 0$ and $w\not\equiv 0$. 
When $\sigma>0$, it was shown that, if  $w \in C_0^\alpha(\mathbb{R}^N) \cap L^1(\mathbb{R}^N)$ for some $\alpha\in (0,1)$, $u_0\in C_0(\mathbb{R}^N)$ and $\int_{\mathbb{R}^N} w(x)\,dx>0$, then  for all $p>1$, the solution to \eqref{eqJKS} blows up in finite time. In the case $-1<\sigma<0$, it was proved that the critical exponent for \eqref{eqJKS} is equal to $\frac{N-2\sigma}{N-2-2\sigma}$ in the following sense:
\begin{itemize}
\item[(a)] if $1<p<\frac{N-2\sigma}{N-2-2\sigma}$ and $\int_{\mathbb{R}^N} w(x)\,dx>0$, then for any $u_0\in C_0(\mathbb{R}^N)$, the solution to \eqref{eqJKS} blows up in finite time;
\item[(b)] if $p\geq \frac{N-2\sigma}{N-2-2\sigma}$, then the solution to  \eqref{eqJKS} exists globally whenever $u_0\in C_0(\mathbb{R}^N)\cap L^d(\mathbb{R}^N)$ and  $w\in L^k(\mathbb{R}^N)$  are such that $\|u_0\|_{L^d}+\|w\|_{L^k}$ is sufficiently small, where $d=\frac{N(p-1)}{2}$ and $k=\frac{d}{p(\sigma+1)-\sigma}\cdot$
\end{itemize}
Motivated by the above contributions, in particular by \cite{JKS}, our goal in this paper is to study the corresponding system to \eqref{eqJKS}, namely 
problem \eqref{1}--\eqref{Initialcondition1}. 

Before stating our main results, we give the 

\begin{defn}[Mild solution]
Let $u_0,v_0,w_1,w_2\in C_0(\mathbb{R}^N)$, $\sigma,\gamma>-1$ and $T>0$.  We say that $(u,v)\in C([0,T], C_0(\mathbb{R}^N) \times C_0(\mathbb{R}^N))$ is a mild solution to 
\eqref{1}--\eqref{Initialcondition1}, if  
\begin{eqnarray}\label{MILD}
\left\{\begin{array}{lll}
 u(t)&=&S(t)u_0+\displaystyle\int_0^t S(t-s)|v(s)|^{p}\,ds+\int_0^t s^\sigma S(t-s)w_1\,ds,\quad 0\leq t\leq T,\\
\displaystyle v(t)&=&S(t)v_0+\displaystyle \int_0^t S(t-s)|u(s)|^{q}\,ds+\int_0^t s^\gamma S(t-s)w_2\,ds,\quad 0\leq t\leq T,
\end{array}
\right.
\end{eqnarray}
where $S(t)=e^{t\Delta}$ is the heat semigroup on $\mathbb{R}^N$.
\end{defn}

We first study the local existence of mild solutions to  \eqref{1}--\eqref{Initialcondition1}. 

\begin{thm}[Local existence] \label{local}
Let $u_0,v_0,w_1,w_2\in C_0(\mathbb{R}^N)$, $\sigma,\gamma>-1$, and $p,q>1$. Then the following holds:
\begin{itemize}
\item[{\rm{(i)}}] There exist $0< T <\infty$  and a unique  mild solution 
$$
(u,v)\in C([0,T],C_0(\mathbb{R}^N)\times C_0(\mathbb{R}^N))
$$ 
to  \eqref{1}--\eqref{Initialcondition1}. 
\item[{\rm{(ii)}}] The solution $(u,v)$ can be extended to a maximal interval  $[0, T_{\max})$, $0<T_{\max}\leq \infty$. Moreover, if $T_{\max}<\infty$, then 
$$
\lim_{t\to T_{\max}^-} \left(\|u\|_{L^\infty((0,t)\times\mathbb{R}^N)}+\|v\|_{L^\infty((0,t)\times\mathbb{R}^N)}\right)=\infty  \mbox{ (finite time blow-up)}.
$$
\item[{\rm{(iii)}}] If, in addition $u_0,v_0,w_1,w_2\in L^r(\mathbb{R}^N)$ for some $1\leq r< \infty$, then 
$$
u,v\in C([0,T_{\max}),C_0(\mathbb{R}^N)\cap L^r(\mathbb{R}^N)).
$$
\end{itemize}
\end{thm}

Next, we study the finite time blow-up of mild solutions to \eqref{1}--\eqref{Initialcondition1}.

\begin{thm}[Blow-up] \label{Blow-up}
Let $\sigma,\gamma>-1$, $\sigma, \gamma\neq 0$. Suppose that $u_0,v_0\in C_0(\mathbb{R}^N)$ and $w_i\in C_0(\mathbb{R}^N)\cap L^1(\mathbb{R}^N)$ are such that  $\displaystyle\int_{\mathbb{R}^N}w_i(x)\,dx>0$,  $i=1,2$. Then the following holds:
\begin{itemize}
\item[{\rm{(i)}}] If $\sigma,\gamma\in(-1,0)$ and
\begin{equation}\label{C1}
\frac{N}{2}< \max\left\{\frac{(\sigma+1)(pq-1)+p+1}{pq-1},\frac{(\gamma+1)(pq-1)+q+1}{pq-1}\right\},
\end{equation}
then the mild solution $(u,v)$ to  \eqref{1}--\eqref{Initialcondition1} blows up in finite time.
\item[{\rm{(ii)}}] If $\sigma>0$ or $\gamma>0$, then for any $p,q>1$, the mild solution $(u,v)$ to  \eqref{1}--\eqref{Initialcondition1} blows up in finite time.
\end{itemize}
\end{thm}

For the proofs of the above blow-up results, we make use of  the test function method  (see e.g. \cite{FinoKirane,MP,Zhang}).

\begin{rem}
Note that no sign conditions are imposed on the initial values in Theorem \ref{Blow-up}.
\end{rem}

We state below the obtained global existence result.  Let us point out that the conditions ensuring the global existence depend on the norms of the initial conditions and forcing terms.

For $\sigma,\gamma\in(-1,0)$ and $p,q>1$, let 
\begin{equation}\label{di}
d_1=\frac{N(pq-1)}{2(p+1)},\quad d_2=\frac{N(pq-1)}{2(q+1)}
\end{equation}
and
\begin{equation}\label{ki}
k_1=\frac{N(pq-1)}{2[(pq-1)(1+\sigma)+p+1]},\quad k_2=\frac{N(pq-1)}{2[(pq-1)(1+\gamma)+q+1]}.
\end{equation}
\begin{thm}[Global existence]\label{global}
Let $\sigma,\gamma\in(-1,0)$,  
$$
u_0\in C_0(\mathbb{R}^N)\cap L^{d_1}(\mathbb{R}^N),\quad v_0\in C_0(\mathbb{R}^N)\cap L^{d_2}(\mathbb{R}^N),\quad w_i\in C_0(\mathbb{R}^N)\cap L^{k_i}(\mathbb{R}^N),\, i=1,2,
$$
and $(u,v)$ be the corresponding mild solution to \eqref{1}--\eqref{Initialcondition1}. If
\begin{equation}\label{C3}
\frac{N}{2}\geq \max\left\{\frac{(\sigma+1)(pq-1)+p+1}{pq-1},\frac{(\gamma+1)(pq-1)+q+1}{pq-1}\right\}
\end{equation}
and 
 \begin{equation}\label{CETR}
p>\max\left\{\frac{\sigma}{\gamma},1\right\},\quad q>\max\left\{\frac{\gamma}{\sigma},1\right\},
\end{equation}
then $(u,v)\in C([0,\infty),C_0(\mathbb{R}^N)\times C_0(\mathbb{R}^N))$ for $\|u_0\|_{L^{d_1}}+\|v_0\|_{L^{d_2}}+\|w_1\|_{L^{k_1}}+\|w_2\|_{L^{k_2}}$ sufficiently small.
\end{thm}

\begin{rem}
Under condition \eqref{C3}, it is still an open question if we have blow-up  or global existence for one of the following cases:
\begin{itemize}
\item[{\rm{(a)}}] $-1<\sigma<\gamma<0$ and $p\leq\frac{\sigma}{\gamma}$.
\item[{\rm{(b)}}] $-1<\gamma<\sigma<0$  and $q\leq\frac{\gamma}{\sigma}$.\end{itemize}
\end{rem}

Consider now the special case of \eqref{1} when $\sigma=\gamma>-1$, $\sigma\neq 0$. Namely,
\begin{eqnarray}\label{1spc}
\left\{
  \begin{array}{lll}
u_t-\Delta u &=&|v|^{p}+t^\sigma w_1(x),\quad x\in\mathbb{R}^N,t>0,\\
v_t-\Delta v &=&|u|^{q}+t^\sigma w_2(x),\quad x\in\mathbb{R}^N,t>0.
  \end{array}
\right.
\end{eqnarray}
Let 
$$
N^*(\sigma,p,q):=2\left(\sigma+1+ \frac{\alpha+1}{pq-1}\right),\quad \alpha=\max\{p,q\}.
$$
From Theorems \ref{Blow-up} and  \ref{global}, one deduces that $N^*(\sigma,p,q)$ is critical for problem  \eqref{1spc}--\eqref{Initialcondition1} in the following sense.

\begin{cor}
Let $\sigma=\gamma\in (-1,0)$ and $p,q>1$. Then the following holds:
\begin{itemize}
\item[{\rm{(i)}}] If $u_0,v_0\in C_0(\mathbb{R}^N)$, $w_i\in C_0(\mathbb{R}^N)\cap L^1(\mathbb{R}^N)$,  $\displaystyle\int_{\mathbb{R}^N}w_i(x)\,dx>0$, $i=1,2$, and
$$
N<N^*(\sigma,p,q),
$$
then the mild solution $(u,v)$ to  \eqref{1spc}--\eqref{Initialcondition1} blows up in finite time.
\item[{\rm{(ii)}}] If $u_0\in C_0(\mathbb{R}^N)\cap L^{d_1}(\mathbb{R}^N)$, $v_0\in C_0(\mathbb{R}^N)\cap L^{d_2}(\mathbb{R}^N)$,  $w_i\in C_0(\mathbb{R}^N)\cap L^{k_i}(\mathbb{R}^N)$, $i=1,2$, and 
$$
N\geq N^*(\sigma,p,q),
$$
then $(u,v)$ exists globally for  $\|u_0\|_{L^{d_1}}+\|v_0\|_{L^{d_2}}+\|w_1\|_{L^{k_1}}+\|w_2\|_{L^{k_2}}$  sufficiently small.
\end{itemize}
\end{cor}

The rest of the paper is organized as follows. In Section \ref{sec2}, we prove the local existence result given by Theorem \ref{local}. The blow-up results stated by Theorem \ref{Blow-up} are proved in Section \ref{sec3}.  The study of the global existence is investigated in Section \ref{sec4}, where we prove Theorem \ref{global}.

\section{Local existence}\label{sec2}
This section is devoted to the proof of Theorem \ref{local}. We first recall some fundamental properties related to $S(t)$, the heat semigroup on $\mathbb{R}^N$ (see e.g. \cite{CH,Evans,Giga}), and fix some notations.  

There exists a constant $C>0$ such that for any $1\leq \zeta \leq \varrho\leq \infty$,   one has
\begin{equation}\label{7}
\|S(t)\vartheta\|_{L^\varrho}\leq C t^{-\frac{N}{2}\left(\frac{1}{\zeta}-\frac{1}{\varrho}\right)} \|\vartheta\|_{L^\zeta}, \quad \vartheta\in L^\zeta(\mathbb{R}^N).
\end{equation}
In particular, one has
\begin{equation}\label{7'}
\|S(t)\vartheta\|_{L^\zeta }\leq  \|\vartheta\|_{L^\zeta}, \quad \vartheta\in L^\zeta(\mathbb{R}^N).
\end{equation}
Furthermore, for all $\vartheta \in C_0(\mathbb{R}^N)$,  it holds that 
$$
\lim_{t\rightarrow 0^+}S(t)\vartheta(x)=\vartheta(x),\quad x\in\mathbb{R}^N.
$$
We denote by $\|\cdot\|_\infty$ the $L^\infty$--norm in $\mathbb{R}^N$. 
Let 
$$
\delta_\infty(f,g)=\max\{\|f\|_\infty,\|g\|_\infty\},\quad f,g\in C_0(\mathbb{R}^N)
$$
and
$$
\delta_r(f,g)=\max\{\|f\|_{L^r},\|g\|_{L^r}\},\quad f,g\in L^r(\mathbb{R}^N),\quad 1\leq r<\infty.
$$

Given $0<T<\infty$, we denote by $\|\cdot\|_{T,\infty}$ the  $L^\infty$--norm in $(0,T)\times \mathbb{R}^N$.

\begin{proof}[Proof of Theorem \ref{local}]
(i)  We first prove uniqueness in the functional space 
$$
C([0,T],C_0(\mathbb{R}^N)\times C_0(\mathbb{R}^N))
$$
for an arbitrary $0<T<\infty$.  Let $(u,v),(\tilde{u},\tilde{v})\in C([0,T],C_0(\mathbb{R}^N)\times C_0(\mathbb{R}^N))$ be two mild solutions
to \eqref{1}--\eqref{Initialcondition1}. Using \eqref{MILD}, property \eqref{7} and the inequality
\begin{equation}\label{usin}
|a^r-b^r|\leq  r \max\{a^{r-1}, b^{r-1}\}|a-b|,\quad r>1,\,\, a, b\geq  0,
\end{equation}
for all $0\leq t\leq T$, one obtains 
$$
\|u(t)-\tilde{u}(t)\|_{\infty}+\|v(t)-\tilde{v}(t)\|_{\infty}
\leq C \int_0^t\left(\|u(s)-\tilde{u}(s)\|_{\infty}+\|v(s)-\tilde{v}(s)\|_{\infty}\right)\,ds.
$$
Next, by Gronwall's inequality, the uniqueness follows.

For arbitrary $0<T<\infty$, we introduce the Banach space $(E_T, \vertiii{\cdot})$ defined by 
$$
E_T=\left\{(u,v)\in C([0,T],C_0(\mathbb{R}^N)\times
C_0(\mathbb{R}^N)):\, \vertiii{(u,v)}\leq
2\left(\delta_\infty(u_0,w_1)+\delta_\infty(v_0,w_2)\right):=2M\right\},
$$
where 
$$
\vertiii{(u,v)}=\|u\|_{T,\infty}+\|v\|_{T,\infty},\quad (u,v)\in E_T.
$$
For every $U=(u,v)\in E_T$, let  $\Psi(U)=\left(\Psi_1(U),\Psi_2(U)\right)$, where
$$
\Psi_1(U)=S(t)u_0+\int_0^tS(t-s)|v(s)|^{p}\,ds+\int_0^ts^\sigma S(t-s)w_1\,ds,\quad t\in[0,T]
$$
and
$$
\Psi_2(U)=S(t)v_0+\int_0^tS(t-s)|u(s)|^{q}\,ds+\int_0^ts^\gamma S(t-s)w_2\,ds,\quad t\in[0,T].
$$
Since $u_0,v_0,w_1,w_2\in C_0(\mathbb{R}^N)$, $\sigma,\gamma>-1$, one can check easily that 
$$
\Psi (E_T)\subset C([0,T],C_0(\mathbb{R}^N)\times C_0(\mathbb{R}^N)).
$$
On the other hand, using \eqref{7'}, for all $U=(u,v)\in E_T$, for all $0<t<T$, one has
\begin{eqnarray*}
  \|\Psi_1(U)\|_\infty &\leq & \|u_0\|_\infty+
  \int_0^t \left\|S(t-s)|v|^{p}\right\|_\infty \,ds+\int_0^t s^\sigma \|w_1\|_\infty \,ds \\
  &\leq & \|u_0\|_\infty+T \|v\|_{T,\infty}^p +\frac{T^{\sigma+1}}{\sigma+1}\|w_1\|_\infty\\
&\leq & \|u_0\|_\infty+T 2^pM^p+\frac{T^{\sigma+1}}{\sigma+1}M\\
     &=& \|u_0\|_\infty+\left(2^pTM^{p-1}+\frac{T^{\sigma+1}}{\sigma+1}\right)M,
  \end{eqnarray*}
which yields
\begin{equation}\label{Fino1}
\|\Psi_1(U)\|_{T,\infty}\leq  \|u_0\|_\infty+\left(2^pTM^{p-1}+\frac{T^{\sigma+1}}{\sigma+1}\right)M.
\end{equation}
Similarly, one obtains
\begin{equation}\label{Fino2}
\|\Psi_2(U)\|_{T,\infty}\leq \|v_0\|_\infty+\left(2^qTM^{q-1}+\frac{T^{\gamma+1}}{\gamma+1}\right)M.
\end{equation}
Combining \eqref{Fino1} with \eqref{Fino2}, one deduces that 
$$
\vertiii{\Psi(U)}\leq M+ 2\max\left\{2^pTM^{p-1}+\frac{T^{\sigma+1}}{\sigma+1}, 2^qTM^{q-1}+\frac{T^{\gamma+1}}{\gamma+1}\right\}M.
$$
Hence, by choosing $0<T\ll1$ small enough so that 
\begin{equation}\label{choix1}
\max\left\{2^pTM^{p-1}+\frac{T^{\sigma+1}}{\sigma+1},2^qTM^{q-1}+\frac{T^{\gamma+1}}{\gamma+1}\right\}\leq \frac{1}{2},
\end{equation}
one obtains $\vertiii{\Psi(U)}\leq 2M$, i.e. 
$$
\Psi(E_T)\subset E_T.
$$
Next, we shall prove that $\Psi: E_T\to E_T$ is a contraction mapping, and by Banach contraction principle, the existence follows.

For $U=(u,v),V=(\tilde{u},\tilde{v})\in E_T$, again using \eqref{7'} and \eqref{usin}, for all $0<t<T$, one has
\begin{eqnarray*}
\|\Psi_1(U)-\Psi_1(V)\|_\infty &\leq & \int_0^t \left\|S(t-s) \left(|v(s)|^p-|\tilde{v}(s)|^p\right)\right\|_\infty\,ds\\
&\leq &  \int_0^t \left\||v(s)|^p-|\tilde{v}(s)|^p\right\|_\infty\,ds\\
&\leq & T 2^{p-1}p M^{p-1} \|v-\tilde{v}\|_{T,\infty},
\end{eqnarray*}
which yields
\begin{equation}\label{AFino1}
\|\Psi_1(U)-\Psi_1(V)\|_{T,\infty}\leq 2^{p-1}p M^{p-1}  T \|v-\tilde{v}\|_{T,\infty}.
\end{equation}
Similarly, one has
\begin{equation}\label{AFino2}
\|\Psi_2(U)-\Psi_2(V)\|_{T,\infty}\leq 2^{q-1}q M^{q-1}  T \|u-\tilde{u}\|_{T,\infty}.
\end{equation}
Combining \eqref{AFino1} with \eqref{AFino2}, it holds that
$$
\vertiii{\Psi(U)-\Psi(V)}\leq  2 \max\left\{2^{p-1}p M^{p-1},2^{q-1}q M^{q-1}\right\} T \vertiii{U-V}.
$$
Therefore, taking $0<T<\infty$ so that \eqref{choix1} is satisfied and 
$$
2 \max\left\{2^{p-1}p M^{p-1},2^{q-1}q M^{q-1}\right\} T <1,
$$
one obtains that $\Psi: E_T\to E_T$ is a contraction mapping.\\
(ii) Using the uniqueness of solutions, we conclude the existence of a maximal interval $[0,T_{\max})$,  where
$$
T_{\max}=\sup\left\{\tau>0:\, (u,v) \mbox{ is a mild solution
to } \eqref{1}-\eqref{Initialcondition1} \mbox{ in } C([0,\tau],C_0(\mathbb{R}^N)\times C_0(\mathbb{R}^N))\right\}.
$$
Furthermore, if $T_{\max}<\infty$, applying similar  arguments as in \cite{JKS}, one  concludes that  
$$
\|u\|_{L^\infty((0,t)\times\mathbb{R}^N)}+\|v\|_{L^\infty((0,t)\times\mathbb{R}^N)}\longrightarrow\infty\quad\mbox{as}\quad  t\rightarrow T_{\max}^-.
$$
(iii) If $u_0,v_0,w_1,w_2\in L^r(\mathbb{R}^N)\cap C_0(\mathbb{R}^N)$ for some $1\leq r<\infty$, repeating the fixed point argument in the functional space
\begin{eqnarray*}
E_{T,r}&=& \{(u,v)\in L^\infty((0,T),(C_0(\mathbb{R}^N)\cap L^r(\mathbb{R}^N))\times( C_0(\mathbb{R}^N)\cap L^r(\mathbb{R}^N))): \\
   &&\vertiii{(u,v)}\leq 2\left(\delta_\infty(u_0,w_1)+\delta_\infty(v_0,w_2)\right),\vertiii{(u,v)}_{r}\leq
  2\left(\delta_\infty(u_0,w_1)+\delta_r(v_0,w_2)\right)\},
\end{eqnarray*}
instead of $E_T$, where
$$
\vertiii{(u,v)}_{r}=\|u\|_{L^\infty((0,T),L^r(\mathbb{R}^N))}+\|v\|_{L^\infty((0,T),L^r(\mathbb{R}^N))},
$$
and  estimating $\|u^p\|_{L^r}$ by
$\|u\|^{p-1}_{\infty}\|u\|_{L^r}$ (the same for $v$) in
the contraction mapping argument, one obtains a unique solution $(u,v)$ in $E_{T,r}$, and we see that 
$$
u,v\in C([0,T_{\max}),C_0(\mathbb{R}^N))\cap C([0,T_{\max}),L^r(\mathbb{R}^N)).
$$
This ends the proof of Theorem \ref{local}.
\end{proof}

\section{Blow-up results}\label{sec3}

In this section, we prove the blow-up results given by Theorem \ref{Blow-up}. First, we give the
 \begin{defn}[Weak solution]\label{definitionweaksolution} 
Let $T>0$, $\sigma,\gamma>-1$, $p,q>1$ and $w_i \in L_{Loc}^1(\mathbb{R}^N)$, $i=1,2$. We say that
$$
(u,v)  \in L^q((0,T),L_{Loc}^q(\mathbb{R}^N))\times  L^p((0,T),L_{Loc}^p(\mathbb{R}^N))
$$
is a weak solution to \eqref{1}--\eqref{Initialcondition1}, if
\begin{equation}\label{weaksolution1}
\begin{aligned}
&\int_0^T\int_{\mathbb{R}^N}|v|^p(t,x)\varphi(t,x)\,dx\,dt+\int_0^T\int_{\mathbb{R}^N}t^\sigma w_1(x)\varphi(t,x)\,dx\,dt\\
&=-\int_0^T\int_{\mathbb{R}^N}u(t,x)\Delta\varphi(t,x)\,dx\,dt-\int_0^T\int_{\mathbb{R}^N}u(t,x)\varphi_t(t,x)\,dx\,dt
\end{aligned}
\end{equation}
and
\begin{equation}\label{weaksolution2}
\begin{aligned}
&\int_0^T\int_{\mathbb{R}^N}|u|^q(t,x)\varphi(t,x)\,dx\,dt+\int_0^T\int_{\mathbb{R}^N}t^\gamma w_2(x)\varphi(t,x)\,dx\,dt\\
&=-\int_0^T\int_{\mathbb{R}^N}v(t,x)\Delta\varphi(t,x)\,dx\,dt
-\int_0^T\int_{\mathbb{R}^N}v(t,x)\varphi_t(t,x)\,dx\,dt,
\end{aligned}
\end{equation}
for all compactly supported test function $\varphi\in C^1([0,T],C^2(\mathbb{R}^N))$ with $\mbox{supp}(\varphi)\subset (0,T)\times \mathbb{R}^N$.
\end{defn}

The following Lemma is crucial in the proof of  Theorem \ref{Blow-up}.

\begin{lem}[Mild$\implies$Weak]\label{MildWeak}
Let $T>0$, $\sigma,\gamma>-1$, $p,q>1$ and $u_0,v_0,w_i\in C_0(\mathbb{R}^N)$, $i=1,2$. If $(u,v)\in C([0,T],C_0(\mathbb{R}^N)\times C_0(\mathbb{R}^N))$ is a mild solution to  \eqref{1}--\eqref{Initialcondition1}, then $(u,v)  \in L^q((0,T),L_{Loc}^q(\mathbb{R}^N))\times  L^p((0,T),L_{Loc}^p(\mathbb{R}^N))$ is a weak solution to  \eqref{1}--\eqref{Initialcondition1}.
\end{lem}

\begin{proof}
The proof of this Lemma is similar to that of  \cite[Lemma 4.2]{FinoKirane}. For the completeness of this paper, we will do it in details.

Let $(u,v)\in C([0,T],C_0(\mathbb{R}^N)\times C_0(\mathbb{R}^N))$ be a mild 
solution to  \eqref{1}--\eqref{Initialcondition1}. Given a compactly supported test function $\varphi\in C^1([0,T],C^2(\mathbb{R}^N))$ with $\mbox{supp}(\varphi)\subset (0,T)\times \mathbb{R}^N$,
 multiplying \eqref{MILD} by $\varphi$ and integrating over $\mathbb{R}^N$, for all $0\leq t\leq T$, one obtains
\begin{eqnarray*}
\int_{\mathbb{R}^N}u(t,x)\varphi(t,x)\,dx &=& \int_{\mathbb{R}^N}S(t)u_0(x)\varphi(t,x)\,dx+ \int_{\mathbb{R}^N}\int_0^tS(t-s)|v(s)|^{p}\,ds\varphi(t,x)\,dx\\
&&+ \int_{\mathbb{R}^N}\int_0^ts^\sigma S(t-s)w_1(x)\,ds\varphi(t,x)\,dx
 \end{eqnarray*}
and
\begin{eqnarray*}
\int_{\mathbb{R}^N}v(t,x)\varphi(t,x)\,dx&=&\int_{\mathbb{R}^N}S(t)v_0(x)\varphi(t,x)\,dx+\int_{\mathbb{R}^N}\int_0^tS(t-s)|u(s)|^{q}\,ds\varphi(t,x)\,dx\\
&&+\int_{\mathbb{R}^N}\int_0^ts^\gamma S(t-s)w_2(x)\,ds\varphi(t,x)\,dx.
\end{eqnarray*}
Differentiating with respect to $t$, it holds that 
\begin{equation}\label{53}
\begin{aligned}
&\frac{d}{dt} \int_{\mathbb{R}^N}u(t,x)\varphi(t,x)\,dx\\
&= \int_{\mathbb{R}^N} \frac{d}{dt} \left(S(t)u_0(x)\varphi(t,x)\right)\,dx+ \int_{\mathbb{R}^N} \frac{d}{dt} \left(\int_0^tS(t-s)|v(s)|^{p}\,ds\varphi(t,x)\right)\,dx\\
&\,\,\,\,\,\,+ \int_{\mathbb{R}^N} \frac{d}{dt} \left(\int_0^ts^\sigma S(t-s)w_1(x)\,ds\varphi(t,x)\right)\,dx
\end{aligned}
\end{equation}
and
\begin{equation}\label{54}
\begin{aligned}
&\frac{d}{dt} \int_{\mathbb{R}^N}v(t,x)\varphi(t,x)\,dx\\
&=\int_{\mathbb{R}^N}\frac{d}{dt} \left(S(t)v_0(x)\varphi(t,x)\right)\,dx+\int_{\mathbb{R}^N}\frac{d}{dt} \left(\int_0^tS(t-s)|u(s)|^{q}\,ds\varphi(t,x)\right)\,dx\\
&\,\,\,\,\,\,+\int_{\mathbb{R}^N}\frac{d}{dt} \left(\int_0^ts^\gamma S(t-s)w_2(x)\,ds\varphi(t,x)\right)\,dx.
\end{aligned}
\end{equation}
As $u_0,w_1, |v(s)|^p\in C_0(\mathbb{R}^N)$, for all $s\in[0,T]$, it follows from the properties of the semigroup $S(t)$ (see \cite[Theorem~1, p.~47]{Evans}) that
\begin{equation}\label{47}
\begin{aligned}
&\int_{\mathbb{R}^N} \frac{d}{dt}\left(S(t)u_0(x)\varphi(t,x)\right)\,dx
\\
&=\int_{\mathbb{R}^N}\Delta\left(S(t)u_0(x)\right)\varphi(t,x)\,dx+ \int_{\mathbb{R}^N} S(t)u_0(x)\varphi_t(t,x)\,dx\\
&=\int_{\mathbb{R}^N} S(t)u_0(x)\Delta\varphi(t,x)\,dx+ \int_{\mathbb{R}^N} S(t)u_0(x)\varphi_t(t,x)\,dx,
\end{aligned}
\end{equation}

\begin{equation}\label{48}
\begin{aligned}
&\int_{\mathbb{R}^N} \frac{d}{dt}\left(\int_0^tS(t-s)|v(s)|^p\,ds\varphi(t,x)\right)\,dx\\
&=\int_{\mathbb{R}^N}|v(t,x)|^p\varphi(t,x)\,dx+\int_{\mathbb{R}^N}\int_0^t\Delta\left(S(t-s)|v(s,x)|^p\right)\,ds\varphi(t,x)\,dx\\
&\,\,\,\,\,\,+\int_{\mathbb{R}^N}\int_0^tS(t-s)|v(s,x)|^p\,ds\varphi_t(t,x)\,dx\\
&=\int_{\mathbb{R}^N}|v(t,x)|^p\varphi(t,x)\,dx+\int_{\mathbb{R}^N}\int_0^tS(t-s)|v(s,x)|^p\,ds\Delta\varphi(t,x)\,dx\\
&\,\,\,\,\,\,+\int_{\mathbb{R}^N}\int_0^tS(t-s)|v(s,x)|^p\,ds\varphi_t(t,x)\,dx
\end{aligned}
\end{equation}
and
\begin{equation}\label{49}
\begin{aligned}
&\int_{\mathbb{R}^N}\frac{d}{dt} \left(\int_0^ts^\sigma S(t-s)w_1(x)\,ds\varphi(t,x)\right)\,dx\\
&=\int_{\mathbb{R}^N}t^\sigma w_1(x)\varphi(t,x)\,dx+\int_{\mathbb{R}^N}\int_0^ts^\sigma\Delta\left(S(t-s)w_1(x)\right)\,ds\varphi(t,x)\,dx\\
&\,\,\,\,\,\,+\int_{\mathbb{R}^N}\int_0^ts^\sigma S(t-s)w_1(x)\,ds\varphi_t(t,x)\,dx\\
&=\int_{\mathbb{R}^N}t^\sigma w_1(x)\varphi(t,x)\,dx+\int_{\mathbb{R}^N}\int_0^ts^\sigma S(t-s)w_1(x)\,ds\Delta\varphi(t,x)\,dx\\
&\,\,\,\,\,\,+\int_{\mathbb{R}^N}\int_0^ts^\sigma S(t-s)w_1(x)\,ds\varphi_t(t,x)\,dx.
\end{aligned}
\end{equation}
Similarly, one has
 \begin{equation}\label{50}
\int_{\mathbb{R}^N} \frac{d}{dt}\left(S(t)v_0(x)\varphi(t,x)\right) \,dx=\int_{\mathbb{R}^N} S(t)v_0(x)\Delta\varphi(t,x)\,dx+ \int_{\mathbb{R}^N} S(t)v_0(x)\varphi_t(t,x)\,dx,
\end{equation}
\begin{equation}\label{51}
\begin{aligned}
&\int_{\mathbb{R}^N} \frac{d}{dt}\left(\int_0^tS(t-s)|u(s)|^q\,ds\varphi(t,x)\right)\,dx\\
&=\int_{\mathbb{R}^N}|u(t,x)|^q\varphi(t,x)\,dx+\int_{\mathbb{R}^N}\int_0^tS(t-s)|u(s,x)|^q\,ds\Delta\varphi(t,x)\,dx\\
&\,\,\,\,\,\,+\int_{\mathbb{R}^N}\int_0^tS(t-s)|u(s,x)|^q\,ds\varphi_t(t,x)\,dx
\end{aligned}
\end{equation}
and
\begin{equation}\label{52}
\begin{aligned}
&\int_{\mathbb{R}^N}\frac{d}{dt} \left(\int_0^ts^\gamma S(t-s)w_2(x)\,ds\varphi(t,x)\right)\,dx\\
&=\int_{\mathbb{R}^N}t^\gamma w_2(x)\varphi(t,x)\,dx+\int_{\mathbb{R}^N}\int_0^ts^\gamma S(t-s)w_2(x)\,ds\Delta\varphi(t,x)\,dx\\
&\,\,\,\,\,\,+\int_{\mathbb{R}^N}\int_0^ts^\gamma S(t-s)w_2(x)\,ds\varphi_t(t,x)\,dx.
\end{aligned}
\end{equation}
Thus, using \eqref{MILD}, \eqref{47}--\eqref{52}, one deduces from \eqref{53} and \eqref{54} that 
\begin{eqnarray*}
\frac{d}{dt} \int_{\mathbb{R}^N}u(t,x)\varphi(t,x)\,dx&=& \int_{\mathbb{R}^N} u(t,x)\Delta\varphi(t,x)\,dx+ \int_{\mathbb{R}^N} u(t,x)\varphi_t(t,x)\,dx\\
&&+ \int_{\mathbb{R}^N}|v(t,x)|^p\varphi(t,x)\,dx+\int_{\mathbb{R}^N}t^\sigma w_1(x)\varphi(t,x)\,dx
\end{eqnarray*}
 and
\begin{eqnarray*}
\frac{d}{dt} \int_{\mathbb{R}^N}v(t,x)\varphi(t,x)\,dx&=& \int_{\mathbb{R}^N} v(t,x)\Delta\varphi(t,x)\,dx+ \int_{\mathbb{R}^N} v(t,x)\varphi_t(t,x)\,dx\\
&&+ \int_{\mathbb{R}^N}|u(t,x)|^q\varphi(t,x)\,dx+\int_{\mathbb{R}^N}t^\gamma w_2(x)\varphi(t,x)\,dx.
\end{eqnarray*}
Finally, integrating in time over  $[0,T]$ and using the fact that  
$\mbox{supp}(\varphi)\subset (0,T)\times \mathbb{R}^N$,  one obtains that $(u,v)$ satisfies \eqref{weaksolution1} and \eqref{weaksolution2}, i.e. $(u,v)$ is a weak solution to  \eqref{1}--\eqref{Initialcondition1}.
\end{proof}

\begin{proof}[Proof of Theorem \ref{Blow-up}]
(i) Suppose, on the contrary, that $(u,v)$ is a global mild solution to \eqref{1}--\eqref{Initialcondition1}. Then, for all $T\gg 1$, 
 $(u,v)\in C([0,T], C_0(\mathbb{R}^N) \times C_0(\mathbb{R}^N))$ solves 
\eqref{MILD}. By Lemma \ref{MildWeak}, one deduces that $(u,v)$ solves \eqref{weaksolution1}-\eqref{weaksolution2}, for all $T\gg 1$ and any compactly supported test function $\varphi\in C^1([0,T],C^2(\mathbb{R}^N))$  with $\mbox{supp}(\varphi)\subset (0,T)\times \mathbb{R}^N$. 

We introduce the cut-off functions $\xi_i\in C^\infty([0,\infty))$, $i=1,2$, satisfying 
$$
\xi_1\geq0,\quad \xi_1\not\equiv 0,\quad \mbox{supp}(\xi_1)\subset (0,1)
$$
and 
$$
0\leq \xi_2\leq 1,\quad  \xi_2\equiv 1 \mbox{ in }[0,1],\quad  \xi_2\equiv 0 \mbox{ in } [2,\infty).
$$
Next, for $T\gg1 $, we take
\begin{equation}\label{testf}
\varphi(t,x)=\varphi_1(t) \varphi_2(x),\quad (t,x)\in [0,T]\times \mathbb{R}^N,
\end{equation}
where 
$$
\varphi_1(t)=\xi_1\left(\frac{t}{T}\right)^\ell,\quad 0\leq t\leq T,
$$

$$
\varphi_2(x)=\xi_2\left(\frac{|x|^2}{T}\right)^\ell,\quad x\in \mathbb{R}^N
$$
and  $\ell\gg 1$. Clearly, $\varphi\in C^1([0,T],C^2(\mathbb{R}^N))$ is a compactly supported function and $\mbox{supp}(\varphi)\subset (0,T)\times \mathbb{R}^N$.  Hence, using 
\eqref{weaksolution1} and \eqref{weaksolution2}, one obtains

\begin{equation}\label{21}
\begin{aligned}
&\int_{\Omega_T}|v|^p(t,x)\varphi(t,x)\,dx\,dt+\int_{\Omega_T}t^\sigma w_1(x)\varphi(t,x)\,dx\,dt\\
&\leq \int_{\Omega_T}|u(t,x)| |\Delta\varphi(t,x)|\,dx\,dt+\int_{\Omega_T}|u(t,x)| |\varphi_t(t,x)|\,dx\,dt
\end{aligned}
\end{equation}
and
\begin{equation}\label{22}
\begin{aligned}
&\int_{\Omega_T}|u|^q(t,x)\varphi(t,x)\,dx\,dt+\int_{\Omega_T}t^\gamma w_2(x)\varphi(t,x)\,dx\,dt\\
&\leq \int_{\Omega_T}|v(t,x)| |\Delta\varphi(t,x)|\,dx\,dt
+\int_{\Omega_T}|v(t,x)||\varphi_t(t,x)|\,dx\,dt,
\end{aligned}
\end{equation}
where  $\Omega_T=(0,T)\times \mathbb{R}^N$. 

We claim that 
\begin{equation}\label{20}
\int_{\Omega_T}t^\sigma w_1(x)\varphi(t,x)\,dx\,dt\geq C\,T^{\sigma+1}\int_{\mathbb{R}^N}w_1(x)\,dx 
\end{equation}
and
\begin{equation}\label{20'}
\int_{\Omega_T}t^\gamma w_2(x)\varphi(t,x)\,dx\,dt\geq C\,T^{\gamma+1}\int_{\mathbb{R}^N}w_2(x)\,dx,
\end{equation}
where we denote by $C$ a positive constant (independent of $T$), whose value may change from line to line.

Indeed, by \eqref{testf}, one has
\begin{equation}\label{16}
\int_{\Omega_T}t^\sigma w_1(x)\varphi(t,x)\,dx\,dt =\left(\int_0^Tt^\sigma\xi_1\left(\frac{t}{T}\right)^\ell\,dt\right)\left(\int_{\mathbb{R}^N}w_1(x)\xi_2\left(\frac{|x|^2}{T}\right)^\ell\,dx\right).
\end{equation}
By the dominated convergence theorem and due to the fact that $w_1\in L^1(\mathbb{R}^N)$, we get
$$
\lim_{T\rightarrow\infty}\int_{\mathbb{R}^N}w_1(x)\xi_2\left(\frac{|x|^2}{T}\right)^\ell\,dx=\int_{\mathbb{R}^N}w_1(x)\,dx>0,
$$
which implies, for sufficiently large $T>0$, that
\begin{equation}\label{17}
\int_{\mathbb{R}^N}w_1(x)\xi_2\left(\frac{|x|^2}{T}\right)^\ell\,dx\geq\frac{1}{2}\int_{\mathbb{R}^N}w_1(x)\,dx.
\end{equation}
On the other hand, one has
\begin{equation}\label{19}
\int_0^T t^\sigma \xi_1\left(\frac{t}{T}\right)^\ell\,dt=T^{\sigma+1}\int_0^1s^\sigma \xi_1^\ell(s)\,ds.
\end{equation}
Combining \eqref{16}, \eqref{17} and \eqref{19}, one obtains \eqref{20}.
Using the same argument, \eqref{20'} follows.

Next, to estimate the right-hand sides of \eqref{21}--\eqref{22}, we use H\"{o}lder's inequality to obtain
\begin{equation}\label{esF1}
\begin{aligned}
&\int_{\Omega_T}|u(t,x)||\varphi_t(t,x)|\,dx\,dt\\
&\leq \left(\int_{\Omega_T}|u(t,x)|^q\varphi(t,x)\,dx\,dt\right)^{\frac{1}{q}}\left(\int_{\Omega_T}\varphi(t,x)^{\frac{-1}{q-1}}|\varphi_t(t,x)|^{\frac{q}{q-1}}\,dx\,dt\right)^{\frac{q-1}{q}}
\end{aligned}
\end{equation}
and
\begin{equation}\label{esF2}
\begin{aligned}
&\int_{\Omega_T}|u(t,x)||\Delta\varphi(t,x)|\,dx\,dt\\
&\leq \left(\int_{\Omega_T}|u(t,x)|^q\varphi(t,x)\,dx\,dt\right)^{\frac{1}{q}}\left(\int_{\Omega_T}\varphi(t,x)^{\frac{-1}{q-1}}|\Delta\varphi(t,x)|^{\frac{q}{q-1}}\,dx\,dt\right)^{\frac{q-1}{q}}.
\end{aligned}
\end{equation}
Similarly, one has
\begin{equation}\label{esFv1}
\begin{aligned}
&\int_{\Omega_T}|v(t,x)||\varphi_t(t,x)|\,dx\,dt\\
&\leq \left(\int_{\Omega_T}|v(t,x)|^p\varphi(t,x)\,dx\,dt\right)^{\frac{1}{p}}\left(\int_{\Omega_T}\varphi(t,x)^{\frac{-1}{p-1}}|\varphi_t(t,x)|^{\frac{p}{p-1}}\,dx\,dt\right)^{\frac{p-1}{p}}
\end{aligned}
\end{equation}
and
\begin{equation}\label{esFv2}
\begin{aligned}
&\int_{\Omega_T}|v(t,x)||\Delta\varphi(t,x)|\,dx\,dt\\
&\leq \left(\int_{\Omega_T}|v(t,x)|^p\varphi(t,x)\,dx\,dt\right)^{\frac{1}{p}}\left(\int_{\Omega_T}\varphi(t,x)^{\frac{-1}{p-1}}|\Delta\varphi(t,x)|^{\frac{p}{p-1}}\,dx\,dt\right)^{\frac{p-1}{p}}.
\end{aligned}
\end{equation}
Hence, using \eqref{21}, \eqref{20}, \eqref{esF1} and \eqref{esF2}, one deduces that 
\begin{equation}\label{25}
\int_{\Omega_T}|v(t,x)|^p\varphi(t,x)\,dx\,dt+C\,T^{\sigma+1}\int_{\mathbb{R}^N}w_1(x)\,dx\leq\left(\int_{\Omega_T}|u(t,x)|^q\varphi\,dx\,dt\right)^{\frac{1}{q}}\mathcal{A},
\end{equation}
where
$$
\mathcal{A}=C\left(\left(\int_{\Omega_T}\varphi(t,x)^{\frac{-1}{q-1}}|\varphi_t(t,x)|^{\frac{q}{q-1}}\,dx\,dt\right)^{\frac{q-1}{q}}
+\left(\int_{\Omega_T}\varphi(t,x)^{\frac{-1}{q-1}}|\Delta\varphi(t,x)|^{\frac{q}{q-1}}\,dx\,dt\right)^{\frac{q-1}{q}}\right). 
$$
Similarly, using \eqref{22}, \eqref{20'}, \eqref{esFv1}, \eqref{esFv2}, one obtains
\begin{equation}\label{26}
\int_{\Omega_T}|u(t,x)|^q\varphi(t,x)\,dx\,dt+C\,T^{\gamma+1}\int_{\mathbb{R}^N}w_2(x)\,dx\leq\left(\int_{\Omega_T}|v(t,x)|^p\varphi\,dx\,dt\right)^{\frac{1}{p}}\mathcal{B},
\end{equation}
where
$$
\mathcal{B}=C\left(\left(\int_{\Omega_T}\varphi(t,x)^{\frac{-1}{p-1}}|\varphi_t(t,x)|^{\frac{p}{p-1}}\,dx\,dt\right)^{\frac{p-1}{p}}
+\left(\int_{\Omega_T}\varphi(t,x)^{\frac{-1}{p-1}}|\Delta\varphi(t,x)|^{\frac{p}{p-1}}\,dx\,dt\right)^{\frac{p-1}{p}}\right). 
$$
Combining \eqref{25} with \eqref{26}, it holds that 
\begin{equation}\label{eqest1}
\mathcal{I}+ T^{\sigma+1}\int_{\mathbb{R}^N}w_1(x)\,dx\leq C\mathcal{I}^{\frac{1}{pq}}\mathcal{B}^{\frac{1}{q}} \mathcal{A}
\end{equation}
and
\begin{equation}\label{eqest2}
\mathcal{J}+ T^{\gamma+1}\int_{\mathbb{R}^N}w_2(x)\,dx\leq C\mathcal{J}^{\frac{1}{pq}}\mathcal{A}^{\frac{1}{p}} \mathcal{B},
\end{equation}
where
$$
\mathcal{I}=\int_{\Omega_T}|v(t,x)|^p\varphi(t,x)\,dx\,dt\quad\mbox{and}\quad  \mathcal{J}=\int_{\Omega_T}|u(t,x)|^q\varphi(t,x)\,dx\,dt.
$$
Next, by Young's inequality, one obtains
\begin{equation}\label{Young1}
C\mathcal{I}^{\frac{1}{pq}}\mathcal{B}^{\frac{1}{q}} \mathcal{A}\leq \frac{1}{pq} \mathcal{I}+ C \mathcal{A}^{\frac{pq}{pq-1}} \mathcal{B}^{\frac{p}{pq-1}}.  
\end{equation}
Similarly,
\begin{equation}\label{Young2}
C\mathcal{J}^{\frac{1}{pq}}\mathcal{A}^{\frac{1}{p}} \mathcal{B}\leq \frac{1}{pq} \mathcal{J}+ C \mathcal{B}^{\frac{pq}{pq-1}} \mathcal{A}^{\frac{q}{pq-1}}.  
\end{equation}
It follows from \eqref{eqest1} and \eqref{Young1} that
\begin{equation}\label{Neweq1}
T^{\sigma+1}\int_{\mathbb{R}^N}w_1(x)\,dx\leq C \mathcal{A}^{\frac{pq}{pq-1}} \mathcal{B}^{\frac{p}{pq-1}}.
\end{equation}
Similarly, using \eqref{eqest2} and \eqref{Young2}, one obtains 
\begin{equation}\label{Neweq2}
T^{\gamma+1}\int_{\mathbb{R}^N}w_2(x)\,dx\leq C \mathcal{B}^{\frac{pq}{pq-1}} \mathcal{A}^{\frac{q}{pq-1}}.
\end{equation}
On the other hand, using \eqref{testf}, for $r>1$, one obtains
$$
\begin{aligned}
&\int_{\Omega_T}\varphi(t,x)^{\frac{-1}{r-1}}|\varphi_t(t,x)|^{\frac{r}{r-1}}\,dx\,dt\\
&= \left(\int_0^T \varphi_1(t)^{\frac{-1}{r-1}}|\varphi_1'(t)|^{\frac{r}{r-1}}\,dt\right) \left(\int_{\mathbb{R}^N}\varphi_2(x)\,dx\right)\\
&=\left(\int_0^T  \xi_1\left(\frac{t}{T}\right)^{\frac{-\ell}{r-1}} \left|\frac{d}{dt} \,\xi_1\left(\frac{t}{T}\right)^\ell\right|^{\frac{r}{r-1}}\,dt\right) \left(\int_{\mathbb{R}^N}\xi_2\left(\frac{|x|^2}{T}\right)^\ell\,dx\right)\\
&=C T^{\frac{-1}{r-1}} \left(\int_0^1 \xi_1(s)^{\ell-\frac{r}{r-1}} |\xi_1'(s)|^{\frac{r}{r-1}}\,ds \right)T^{\frac{N}{2}} \left(\int_{\mathbb{R}^N}\xi_2(|y|^2)^\ell\,dy\right)\\
&=CT^{\frac{N}{2}-\frac{1}{r-1}},
\end{aligned}
$$
which yields
\begin{equation}\label{bahi1}
\left(\int_{\Omega_T}\varphi(t,x)^{\frac{-1}{r-1}}|\varphi_t(t,x)|^{\frac{r}{r-1}}\,dx\,dt\right)^{\frac{r-1}{r}}= C T^{\frac{N}{2}\left(\frac{r-1}{r}\right)-\frac{1}{r}},\quad r>1.
\end{equation}
Using similar calculations, one obtains
\begin{equation}\label{bahi2}
\left(\int_{\Omega_T}\varphi(t,x)^{\frac{-1}{r-1}}|\Delta\varphi(t,x)|^{\frac{r}{r-1}}\,dx\,dt\right)^{\frac{r-1}{r}}=CT^{\frac{N}{2}\left(\frac{r-1}{r}\right)-\frac{1}{r}},\quad r>1.
\end{equation}
Taking $r=q$ in \eqref{bahi1}--\eqref{bahi2}, it holds
\begin{equation}\label{esA}
\mathcal{A}=C T^{\frac{N}{2}\left(\frac{q-1}{q}\right)-\frac{1}{q}}.
\end{equation}
Similarly, taking $r=p$ in \eqref{bahi1}--\eqref{bahi2}, it holds
\begin{equation}\label{esB}
\mathcal{B}=C T^{\frac{N}{2}\left(\frac{p-1}{p}\right)-\frac{1}{p}}.
\end{equation}
Next, \eqref{Neweq1}, \eqref{esA} and \eqref{esB} yield
$$
T^{\sigma+1}\int_{\mathbb{R}^N}w_1(x)\,dx\leq C T^{\frac{N}{2}-\frac{p+1}{pq-1}},
$$
i.e.
\begin{equation}\label{puissance1}
\int_{\mathbb{R}^N}w_1(x)\,dx\leq C  T^{\theta_1},
\end{equation}
where
$$
\theta_1=\frac{N}{2}-\frac{p(q+1)}{pq-1}-\sigma.
$$
Similarly, \eqref{Neweq2}, \eqref{esA} and \eqref{esB} yield
\begin{equation}\label{puissance2}
\int_{\mathbb{R}^N}w_2(x)\,dx\leq C  T^{\theta_2},
\end{equation}
where
$$
\theta_2=\frac{N}{2}-\frac{q(p+1)}{pq-1}-\gamma.
$$
Note that \eqref{C1} is equivalent to $\theta_1<0$ or $\theta_2<0$. If $\theta_1<0$, passing to the limit as $T\to \infty$ in \eqref{puissance1}, one obtains $\displaystyle \int_{\mathbb{R}^N}w_1(x)\,dx\leq 0$. If $\theta_2<0$, passing to the limit as $T\to \infty$ in \eqref{puissance2}, one obtains $\displaystyle \int_{\mathbb{R}^N}w_2(x)\,dx\leq 0$. Therefore, in both cases, one obtains a contradiction with the fact that $\displaystyle \int_{\mathbb{R}^N}w_i(x)\,dx>0$, $i=1,2$.  This completes the proof of  part (i) of Theorem \ref{Blow-up}.\\
(ii) As in the proof of (i), we argue by contradiction by supposing that $(u,v)$ is a global mild solution to \eqref{1}--\eqref{Initialcondition1}. Repeating the same calculations made previously using the test function \eqref{testf} with 
$$
\varphi_2(x)=\xi_2\left(\frac{|x|^2}{R^2}\right)^\ell,\quad x\in \mathbb{R}^N,
$$
where  $1\ll R<T$ is a large positive constant independent on $T$, one obtains
\begin{equation}\label{fesaya1}
\int_{\mathbb{R}^N}w_1(x)\,dx\leq C\left(T^{-\sigma-\frac{p(q+1)}{pq-1}}R^N+T^{-\sigma-\frac{p}{pq-1}}R^{N-\frac{2pq}{pq-1}}+T^{-\sigma-\frac{pq}{pq-1}}R^{N-\frac{2p}{pq-1}}+T^{-\sigma}R^{N-\frac{2pq}{pq-1}}\right)
\end{equation}
and
\begin{equation}\label{fesaya2}
\int_{\mathbb{R}^N}w_2(x)\,dx\leq C\left(T^{-\gamma-\frac{q(p+1)}{pq-1}}R^N+T^{-\gamma-\frac{q}{pq-1}}R^{N-\frac{2pq}{pq-1}}+T^{-\sigma-\frac{pq}{pq-1}}R^{N-\frac{2q}{pq-1}}+T^{-\gamma}R^{N-\frac{2pq}{pq-1}}\right).
\end{equation}
If $\sigma>0$, fixing $R$ and passing to the limit as $T\rightarrow\infty$
in \eqref{fesaya1}, one obtains $\displaystyle \int_{\mathbb{R}^N}w_1(x)\,dx\leq 0$. Similarly, if  $\gamma>0$, fixing $R$ and passing to the limit as $T\rightarrow\infty$
in \eqref{fesaya2}, one obtains $\displaystyle \int_{\mathbb{R}^N}w_2(x)\,dx\leq 0$.  Hence, in both cases, one has a contradiction with the fact that $\displaystyle \int_{\mathbb{R}^N}w_i(x)\,dx>0$, $i=1,2$.  This completes the proof of   Theorem \ref{Blow-up}.
\end{proof}

\section{Global existence}\label{sec4}

In this section, we give the proof of the global existence result stated by Theorem \ref{global}. Jsut before, we need some Lemmas.

\begin{lem}\label{lemmaouf1}
Let $\sigma,\gamma \in (-1,0)$ and $p>1$. If 
\begin{equation}\label{GFI1}
N\geq \frac{2[(pq-1)(1+\gamma)+q+1]}{pq-1}\quad\mbox{and}\quad q>\max\left\{\frac{\gamma}{\sigma},1\right\},
\end{equation}
then
\begin{equation}\label{39}
2q[\sigma(pq-1)+p+1] -N(pq-1)<0,
\end{equation}
\begin{equation}\label{45}
2q(p+1) -(N+2q)(pq-1)<0
\end{equation}
and
\begin{equation}\label{46}
2q[\gamma p(pq-1)+p+1] -N(pq-1)<0.
\end{equation}
\end{lem}

\begin{proof}
Let 
$$
g(s)=2q[\sigma(pq-1)+p+1] -s(pq-1),\quad s\in \mathbb{R}.
$$
Since $g'(s)=-(pq-1)<0$, $s\in \mathbb{R}$, one deduces that $g$ is a decreasing function. Hence, from \eqref{GFI1}, it holds that
$$
\begin{aligned}
&2q[\sigma(pq-1)+p+1] -N(pq-1)\\
&=g(N)\\
&\leq  g \left(\frac{2[(pq-1)(1+\gamma)+q+1]}{pq-1}\right)\\
&= 2(pq-1)(q\sigma-\gamma)+2(1-q)\\
&<0,
\end{aligned}
$$
which proves \eqref{39}. Inequalities \eqref{45} and \eqref{46} follow using a similar argument. 
\end{proof}

Similarly, one has the 

\begin{lem}\label{lemmaouf2}
Let $\sigma,\gamma \in (-1,0)$ and $q>1$. If 
$$
N\geq \frac{2[(pq-1)(1+\sigma)+p+1]}{pq-1}\quad\mbox{and}\quad p>\max\left\{\frac{\sigma}{\gamma},1\right\},
$$
then
\begin{equation}\label{40}
2p[\gamma(pq-1)+q+1] -N(pq-1)<0,
\end{equation}
\begin{equation}\label{43}
2p(q+1) -(N+2p)(pq-1)<0
\end{equation}
and
\begin{equation}\label{44}
2p[\sigma q(pq-1)+q+1] -N(pq-1)<0.
\end{equation}
\end{lem}

\begin{lem}\label{lemmaouff}
Let $\sigma,\gamma \in (-1,0)$. If \eqref{C3}--\eqref{CETR} are satisfied, then there exist $q_1,q_2>0$ such that 
\begin{eqnarray}\label{30}
\left\{\begin{array}{llll}
\displaystyle \max\left\{\frac{1}{qd_2},\frac{1}{k_1}-\frac{2}{N}\right\}<\frac{1}{q_1}<\min\left\{\frac{1}{d_1},\frac{1}{q}\right\},\\ \\
\displaystyle \max\left\{\frac{1}{pd_1},\frac{1}{k_2}-\frac{2}{N}\right\}<\frac{1}{q_2}<\min\left\{\frac{1}{d_2},\frac{1}{p}\right\},
\end{array}
\right.
\end{eqnarray}
and 
\begin{eqnarray}\label{31}
\left\{\begin{array}{lll}
\displaystyle \frac{p}{q_2}-\frac{1}{q_1}<\frac{2}{N},\\ \\
\displaystyle \frac{q}{q_1}-\frac{1}{q_2}<\frac{2}{N},
\end{array}
\right.
\end{eqnarray}
where $k_i, d_i$, $i=1,2$, are defined by \eqref{di} and \eqref{ki}.
\end{lem}

\begin{proof}
Let
$$
\alpha_1= \max\left\{\frac{1}{qd_2},\frac{1}{k_1}-\frac{2}{N}\right\},\quad \alpha_2=\min\left\{\frac{1}{d_1},\frac{1}{q}\right\}
$$
and
$$\alpha_3= \max\left\{\frac{1}{pd_1},\frac{1}{k_2}-\frac{2}{N}\right\},\quad \alpha_4=\min\left\{\frac{1}{d_2},\frac{1}{p}\right\}.
$$
First, we shall prove that $0<\alpha_1<\alpha_2$ and $0<\alpha_3<\alpha_4$. Since $\frac{1}{qd_2}>0$, then $\alpha_1>0$. As $p,q>1$ and $d_2>1$, one deduces that $\frac{1}{qd_2}<\alpha_2$. Moreover, $\sigma<0$ implies that $\frac{1}{k_1}-\frac{2}{N}<\frac{1}{d_1}$. Furthermore, $\frac{1}{k_1}-\frac{2}{N}<\frac{1}{q}$ follows from \eqref{39}. This proves that $0<\alpha_1<\alpha_2$. Similarly, since $\frac{1}{pd_1}>0$, then $\alpha_3>0$. As $p,q>1$ and $d_1>1$, one can see that $\frac{1}{pd_1}<\alpha_4$. Moreover, $\gamma<0$ implies that $\frac{1}{k_2}-\frac{2}{N}<\frac{1}{d_2}$. Furthermore, $\frac{1}{k_2}-\frac{2}{N}<\frac{1}{p}$ follows from \eqref{40}. This proves that $0<\alpha_3<\alpha_4$. Then
$$
\Lambda:=\left\{(a,b)\in \mathbb{R}^2:\, 0<\alpha_1<a<\alpha_2,\, 0<\alpha_3<b<\alpha_4\right\}\neq \emptyset.
$$ 
Next, we have to show that 
\begin{equation}\label{aim}
\Lambda \cap \Phi \neq \emptyset,
\end{equation}
where
$$
\Phi:=\left\{(a,b)\in [0,1]\times [0,1]:\, 0<pb-a<\frac{2}{N},\, 0<qa-b<\frac{2}{N}\right\}.
$$
Note that $\Lambda$ is a rectangle in $[0,1]\times[0,1]$ with vertices $(\alpha_1,\alpha_3)$, $(\alpha_2,\alpha_3)$, $(\alpha_2,\alpha_4)$, $(\alpha_1,\alpha_4)$, while $\Phi$ is a quadrilateral in $[0,1]\times[0,1]$ with vertices $(0,0)$, $(\frac{2}{qN},0)$, $(\frac{1}{d_1},\frac{1}{d_2})$, $(0,\frac{2}{pN})$. To prove \eqref{aim}, we have to distinguish four cases.\\
$\bullet$ Case 1: $(\alpha_2,\alpha_4)=\left(\frac{1}{d_1},\frac{1}{d_2}\right)$ (see Figure ~\ref{figcase1}). As the two vertices on the right of the two areas $\Lambda$ and $\Phi$ are the same, it is obvious that $\Lambda \cap \Phi \neq \emptyset$. 
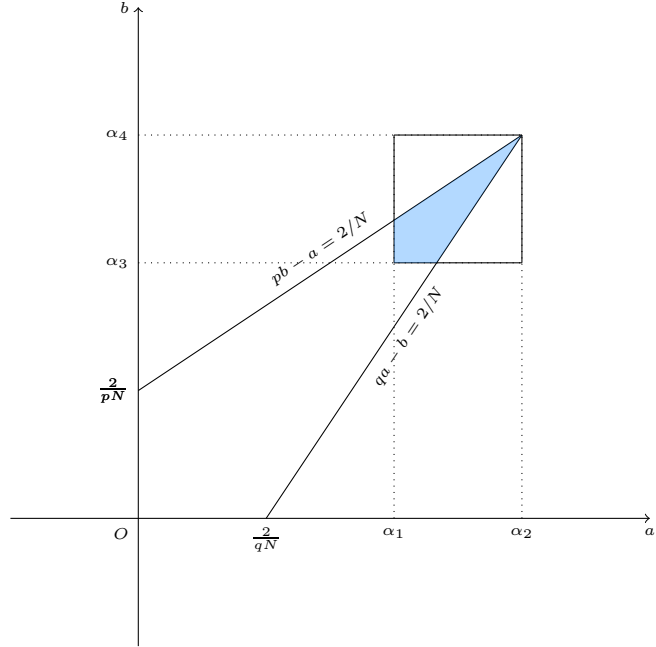
\begin{figure}[htbp]
\begin{tikzpicture}[scale=1.7]
 \draw[->] (-1,0) -- (4,0) node[below]{\tiny$a$};
 \draw[->] (0,-1) -- (0,4) node[left]{\tiny$b$};
 \draw(0,0) node[below left]{\tiny$O$};
 \draw(1,0) node[below]{\tiny$\frac{2}{qN}$}--node[below,sloped] {\tiny$qa-b=2/N$}(3,3);
 \draw[dotted](3,0)node[below]{\tiny$\alpha_2$}--(3,3);
 \draw(0,1) node[left]{\tiny$\frac{2}{pN}$};
  \draw(0,1) node[left]{\tiny$\frac{2}{pN}$}--node[above,sloped] {\tiny$pb-a=2/N$}(3,3);
   \draw[dotted](0,3)node[left]{\tiny$\alpha_4$}--(3,3);
      \draw[dotted](2,0)node[below]{\tiny$\alpha_1$}--(2,3);
  \draw[dotted](0,2)node[left]{\tiny$\alpha_3$}--(3,2);
    \draw(2,2)--(3,2);
    \draw(2,2)--(2,3);
     \draw(3,2)--(3,3);
    \draw(2,3)--(3,3);
     \fill[blue!50!cyan,opacity=0.3](2,2.33333)--(3,3)--(2.33333,2)--(2,2);
 \end{tikzpicture}
 \caption{The region $\Lambda\cap \Phi$ (case 1)}\label{figcase1}
\end{figure}
Namely, one can take
$$
\frac{N\alpha_3+2}{Nq}<a<\frac{N\alpha_1+2(p+1)}{Npq}\quad\mbox{and}\quad  \frac{N\alpha_1+2}{Np}<b<\frac{N\alpha_3+2(q+1)}{Npq}.
$$
\noindent $\bullet$ Case 2: $\alpha_2<\frac{N\alpha_4+2}{Nq}$ and $\alpha_4<\frac{N\alpha_2+2}{Np}$. Graphically (see Figure~\ref{figcase2}), it is obvious that the two areas $\Lambda$ and $\Phi$ intersect at least for one point. Namely, we can find it analytically. For example, let $(a,b)$ be such that 
$$ 
\max\left\{\alpha_1, p\alpha_4-\frac{2}{N}\right\}<a<\alpha_2\quad\mbox{and}\quad  \max\left\{\alpha_3,q\alpha_2-\frac{2}{N}\right\}<b<\alpha_4.
$$
Then $qa<q\alpha_2$ and $-b<-q\alpha_2+\frac{2}{N}$ (resp. $pb<p\alpha_4$ and $-a<-p\alpha_4+\frac{2}{N}$), which yield  $qa-b<\frac{2}{N}$ (resp. $pb-a<\frac{2}{N}$).

\begin{figure}[htbp]
\begin{tikzpicture}[scale=1.6]
 \draw[->] (-1,0) -- (4,0) node[below]{\tiny$a$};
 \draw[->] (0,-1) -- (0,4) node[left]{\tiny$b$};
 \draw(0,0) node[below left]{\tiny$O$};
 \draw(1,0) node[below]{\tiny$\frac{2}{qN}$}--node[below,sloped] {\tiny$qa-b=2/N$}(3,3);
 \draw[dotted](3,0)node[below]{\tiny$d_1^{-1}$}--(3,3);
 \draw(0,1) node[left]{\tiny$\frac{2}{pN}$};
  \draw(0,1) node[left]{\tiny$\frac{2}{pN}$}--node[above,sloped] {\tiny$pb-a=2/N$}(3,3);
   \draw[dotted](0,3)node[left]{\tiny$d_2^{-1}$}--(3,3);
      \draw[dotted](1.5,0)node[below]{\tiny$\alpha_1$}--(1.5,2.5);
        \draw[dotted](2.5,0)node[below]{\tiny$\alpha_2$}--(2.5,2.5);
  \draw[dotted](0,1.5)node[left]{\tiny$\alpha_3$}--(2.5,1.5);     
    \draw[dotted](0,2.5)node[left]{\tiny$\alpha_4$}--(2.5,2.5);   
     \draw(1.5,2.5)--(2.5,2.5);
    \draw(2.5,1.5)--(2.5,2.5);
     \draw(1.5,1.5)--(2.5,1.5);
        \draw(1.5,1.5)--(1.5,2.5);
        \fill[blue!50!cyan,opacity=0.3](1.5,2)--(2.25,2.5)--(2.5,2.5)--(2.5,2.25)--(2,1.5)--(1.5,1.5);
        \draw[->](2.666666,0)--(2.7,-0.3)node[below right]{\tiny$\frac{N\alpha_4+2}{Nq}$};
      \draw[dotted](2.5,2.5)--(2.666666,2.5);
  \draw[dotted](2.666666,2.5)--(2.666666,0);       
   \draw[dotted](2.5,2.5)--(2.5,2.666666);
  \draw[dotted](2.5,2.666666)--(0,2.666666); 
 \draw[->](0,2.666666)--(-0.3,2.7)node[left]{\tiny$\frac{N\alpha_2+2}{Np}$};      
 \end{tikzpicture}
 \caption{The region $\Lambda\cap \Phi$ (case 2)}\label{figcase2}
\end{figure}
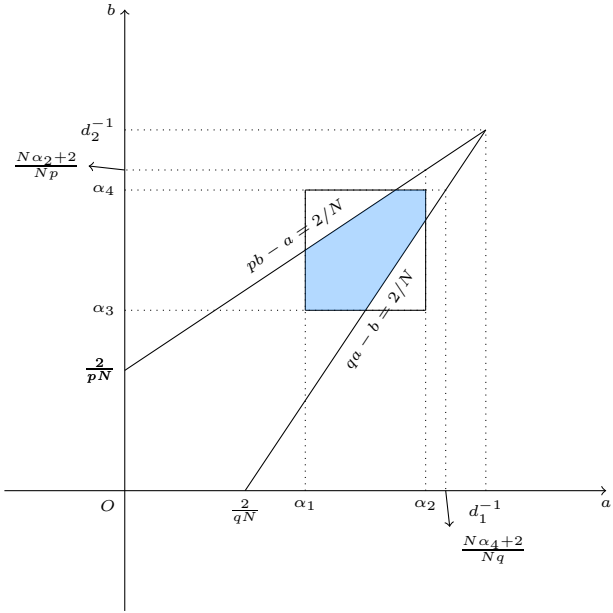
\newpage
\noindent $\bullet$ Case 3: $\alpha_2\geq\frac{N\alpha_4+2}{Nq}$ and $\alpha_4<\frac{N\alpha_2+2}{Np}$. By \eqref{43}--\eqref{44}, one can check that $\alpha_1<\frac{N\alpha_4+2}{Nq}$. So graphically (see Figure~\ref{figcase3}),  one can see that $\Lambda\cap \Phi\neq \emptyset$. Namely, one can take $(a,b)$ such that 
$$
\max\left\{\alpha_1,p\alpha_4-\frac{2}{N}\right\}<a<\frac{Nb+2}{Nq}
$$
and
$$
\max\left\{\alpha_3,q\max\left\{\alpha_1,p\alpha_4-\frac{2}{N}\right\}-\frac{2}{N}\right\}<b<\alpha_4.
$$

\begin{figure}[htbp]
\begin{tikzpicture}[scale=1.7]
 \draw[->] (-1,0) -- (5,0) node[below]{\tiny$a$};
 \draw[->] (0,-1) -- (0,5) node[left]{\tiny$b$};
 \draw(0,0) node[below left]{\tiny$O$};
 \draw(1,0) node[below]{\tiny$\frac{2}{qN}$}--node[below,sloped] {\tiny$qa-b=2/N$}(4,4);
 \draw[dotted](4,0)node[below]{\tiny$d_1^{-1}$}--(4,4);
 \draw(0,1) node[left]{\tiny$\frac{2}{pN}$};
  \draw(0,1) node[left]{\tiny$\frac{2}{pN}$}--node[above,sloped] {\tiny$pb-a=2/N$}(4,4);
   \draw[dotted](0,4)node[left]{\tiny$d_2^{-1}$}--(4,4);
      \draw[dotted](2.5,0)node[below]{\tiny$\alpha_1$}--(2.5,2.5);
        \draw[dotted](3.5,0)node[below]{\tiny$\alpha_2$}--(3.5,2.5);
  \draw[dotted](0,1.5)node[left]{\tiny$\alpha_3$}--(3,1.5);     
    \draw[dotted](0,2.5)node[left]{\tiny$\alpha_4$}--(3,2.5);   
     \draw(2.5,2.5)--(3.5,2.5);
    \draw(3.5,2.5)--(3.5,1.5);
     \draw(2.5,2.5)--(2.5,1.5);
        \draw(2.5,1.5)--(3.5,1.5);
         \draw[->](2.875,0)--(2.9,-0.3)node[below right]{\tiny$\frac{N\alpha_4+2}{Nq}$};
      \draw[dotted](2.875,2.5)--(2.875,0);
    \draw[dotted](3.5,2.5)--(3.5,3.625);
   \draw[dotted](3.5,3.625)--(0,3.625);
  \draw[->](0,3.625)--(-0.3,3.7)node[left]{\tiny$\frac{N\alpha_2+2}{Np}$};     
   \fill[blue!50!cyan,opacity=0.3](2.5,2.5)--(2.875,2.5)--(2.5,2);
 \end{tikzpicture}
 \caption{The region $\Lambda\cap \Phi$ (case 3)}\label{figcase3}
\end{figure}
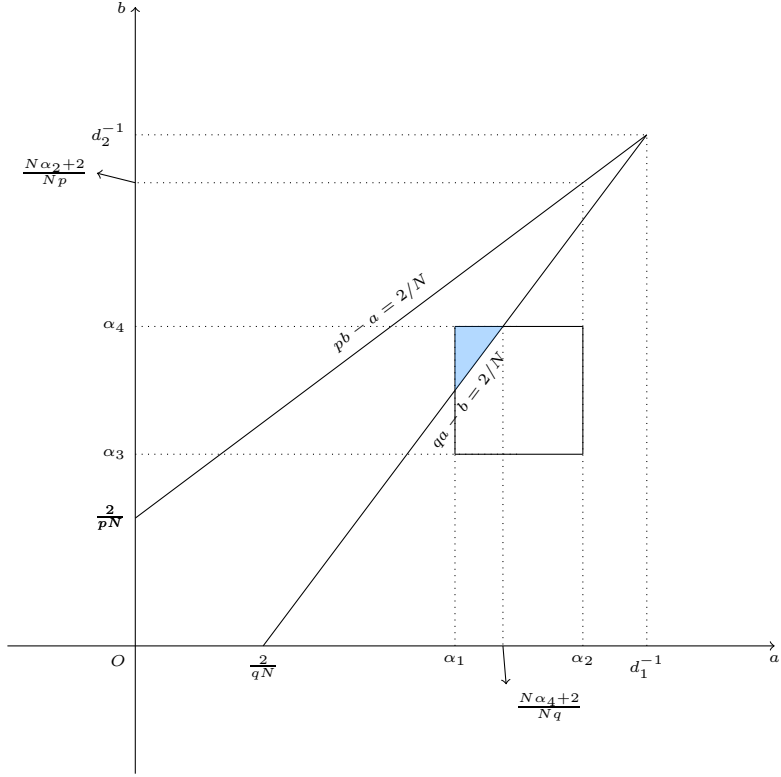

\noindent$\bullet$ Case 4: $\alpha_4\geq\frac{N\alpha_2+2}{Np}$ and $\alpha_2<\frac{N\alpha_4+2}{Nq}$. By \eqref{45}--\eqref{46}, one can check that $\alpha_3<\frac{N\alpha_2+2}{Np}$. So graphically (see Figure~\ref{figcase4}), one can see that $\Lambda\cap \Phi\neq \emptyset$. Namely, one can take $(a,b)$ such that 
$$
\max\left\{\alpha_1,p\max\left\{\alpha_3,q\alpha_2-\frac{2}{N}\right\}-\frac{2}{N}\right\}<a<\alpha_2
$$
and
$$
\max\left\{\alpha_3,q\alpha_2-\frac{2}{N}\right\}<b<\frac{Na+2}{Np}.
$$
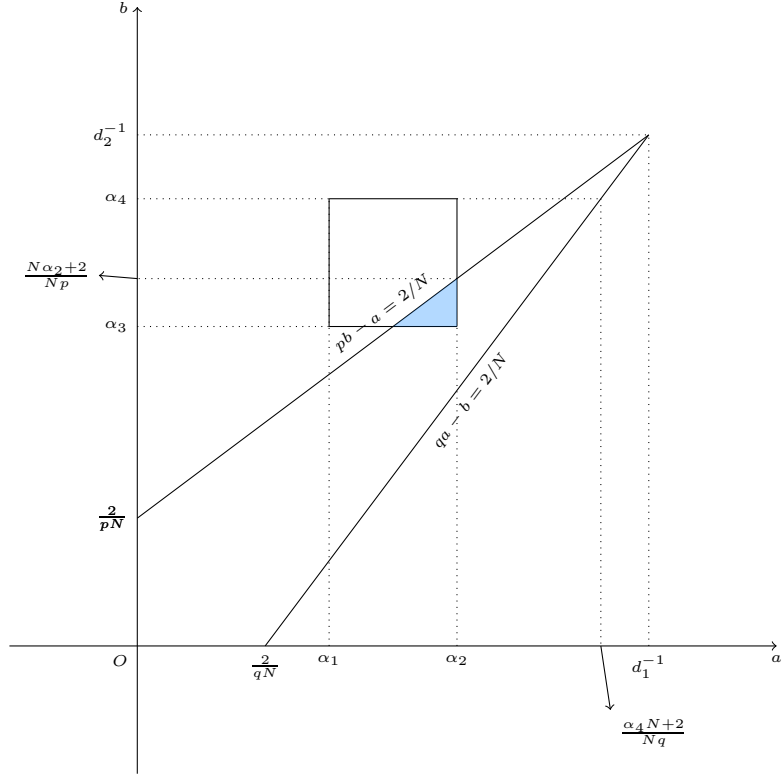
\begin{figure}[htbp]
\begin{tikzpicture}[scale=1.7]
 \draw[->] (-1,0) -- (5,0) node[below]{\tiny$a$};
 \draw[->] (0,-1) -- (0,5) node[left]{\tiny$b$};
 \draw(0,0) node[below left]{\tiny$O$};
 \draw(1,0) node[below]{\tiny$\frac{2}{qN}$}--node[below,sloped] {\tiny$qa-b=2/N$}(4,4);
 \draw[dotted](4,0)node[below]{\tiny$d_1^{-1}$}--(4,4);
 \draw(0,1) node[left]{\tiny$\frac{2}{pN}$};
  \draw(0,1) node[left]{\tiny$\frac{2}{pN}$}--node[above,sloped] {\tiny$pb-a=2/N$}(4,4);
   \draw[dotted](0,4)node[left]{\tiny$d_2^{-1}$}--(4,4);
      \draw[dotted](1.5,0)node[below]{\tiny$\alpha_1$}--(1.5,3.5);
        \draw[dotted](2.5,0)node[below]{\tiny$\alpha_2$}--(2.5,2.5);
  \draw[dotted](0,2.5)node[left]{\tiny$\alpha_3$}--(1.5,2.5);     
    \draw[dotted](0,3.5)node[left]{\tiny$\alpha_4$}--(1.5,3.5);   
     \draw(2.5,2.5)--(2.5,3.5);
    \draw(1.5,2.5)--(1.5,3.5);
     \draw(1.5,2.5)--(2.5,2.5);
        \draw(1.5,3.5)--(2.5,3.5);
     \draw[dotted](2.5,2.875)--(0,2.875);
   \draw[->](0,2.875)--(-0.3,2.9)node[left]{\tiny$\frac{N\alpha_2+2}{Np}$}; 
  \draw[dotted](2.5,3.5)--(3.625,3.5); 
  \draw[dotted](3.625,3.5)--(3.625,0); 
    \draw[->](3.625,0)--(3.7,-0.5)node[below right]{\tiny$\frac{\alpha_4 N+2}{Nq}$};     
     \fill[blue!50!cyan,opacity=0.3] (2.5,2.5)--(2,2.5)--(2.5,2.875);      
 \end{tikzpicture}
 \caption{The region $\Lambda\cap \Phi$ (case 4)}\label{figcase4}
\end{figure}
Hence, in all  cases, we proved that \eqref{aim} holds. This completes the proof of Lemma \ref{lemmaouff}.
\end{proof}

\newpage
Now, we are ready to prove the global existence.

\begin{proof}[Proof of Theorem \ref{global}]
Let $q_1,q_2>0$ be two constants satisfying \eqref{30}--\eqref{31}. Note that by Lemma \ref{lemmaouff}, such constants exist.  It follows that
\begin{equation}\label{33}
 q_1>q,\,\,  q_1>d_1>k_1\geq 1\quad\mbox{and}\quad  q_2>p,\,\,q_2>d_2>k_2\geq 1.
\end{equation}
Let
$$
\beta_1=\frac{N}{2}\left(\frac{1}{d_1}-\frac{1}{q_1}\right),\quad\beta_2=\frac{N}{2}\left(\frac{1}{d_2}-\frac{1}{q_2}\right),
$$
$$
\beta_3=\frac{N}{2}\left(\frac{p}{q_2}-\frac{1}{q_1}\right),\quad \beta_4=\frac{N}{2}\left(\frac{q}{q_1}-\frac{1}{q_2}\right)
$$
and
$$\beta_5=\frac{N}{2}\left(\frac{1}{k_1}-\frac{1}{q_1}\right),\quad \beta_6=\frac{N}{2}\left(\frac{1}{k_2}-\frac{1}{q_2}\right).
$$
 From \eqref{30}--\eqref{31}, one can easily see that
\begin{equation}\label{32}
  0<\beta_i\, (i=1,2), \quad 0<\beta_i<1\, (3\leq i\leq 6), \quad \beta_1q<1\quad\mbox{and}\quad\beta_2p<1.
 \end{equation}
 Moreover, one has
\begin{equation}\label{36}
\beta_1+\sigma-\beta_5+1=\beta_2+\sigma-\beta_6+1=\beta_1-\beta_3-\beta_2p+1=\beta_2-\beta_4-\beta_1q+1=0.
\end{equation}
As $u_0\in L^{d_1}$ and $v_0\in L^{d_2}$, using \eqref{7}, we get
\begin{eqnarray}\label{34}
\left\{\begin{array}{lll}
 \displaystyle \sup_{t>0}t^{\beta_1}\|S(t)u_0\|_{L^{q_1}}\leq  C \|u_0\|_{L^{d_1}}=\eta_1<\infty,\\ \\
\displaystyle \sup_{t>0}t^{\beta_2}\|S(t)v_0\|_{L^{q_2}}\leq C \|v_0\|_{L^{d_2}}=\eta_2<\infty.
\end{array}
\right.
\end{eqnarray}
Set
\begin{equation}\label{35}
    \Xi=\left\{U=(u,v)\in
    L^\infty\left((0,\infty),L^{q_1}(\mathbb{R}^N)\times L^{q_2}(\mathbb{R}^N)\right):\;\mbox{ess}\sup_{t>0}\left(t^{\beta_1}\|u(t)\|_{L^{q_1}}+t^{\beta_2}\|v(t)\|_{L^{q_2}}\right)\leq k\right\},
\end{equation}
where $k>0$ is to be chosen sufficiently small. We endow $\Xi$ with the metric
$$
d_{\Xi}(U,V)=\mbox{ess}\sup_{t>0}\left\{t^{\beta_1}\|u_1(t)-v_1(t)\|_{L^{q_1}}+t^{\beta_2}\|u_2(t)-v_2(t)\|_{L^{q_2}}\right\},
$$
for all $U=(u_1,u_2), V=(v_1,v_2)\in\Xi$. Then $(\Xi,d_{\Xi})$ is a complete metric space.  Let
$$
\Theta(U)=(\Theta_1(U),\Theta_2(U)),\quad U\in \Xi,
$$
where
\begin{eqnarray*}
    \Theta_1(U)(t)&=&S(t)u_0+\int_0^tS(t-s)|v(s)|^{p}\,ds+\int_0^ts^\sigma S(t-s)w_1\,ds.\\
    \Theta_2(U)(t)&=&S(t)v_0+\int_0^tS(t-s)|u(s)|^{q}\,ds+\int_0^ts^\sigma S(t-s)w_2\,ds,
\end{eqnarray*}
for  $t\geq 0$, a.e. By \eqref{7}, \eqref{33}, \eqref{32}, \eqref{36}, \eqref{34} and \eqref{35}, one has
\begin{eqnarray*}
t^{\beta_1}\|\Theta_1(U)(t)\|_{L^{q_1}}&\leq & \eta_1+Ct^{\beta_1}\int_0^t(t-s)^{-\beta_3} \|v(s)\|^p_{L^{q_2}}\,ds+Ct^{\beta_1}\int_0^ts^\sigma(t-s)^{-\beta_5} \|w_1\|_{L^{k_1}}\,ds\\
    & =& \eta_1+Ct^{\beta_1}\int_0^t(t-s)^{-\beta_3} \|v(s)\|^p_{L^{q_2}}\,ds+C t^{\beta_1+\sigma-\beta_5+1}\|w_1\|_{L^{k_1}}\\
     &\leq& \eta_1+Ck^pt^{\beta_1}\int_0^t(t-s)^{-\beta_3} s^{-\beta_2p}\,ds+C \|w_1\|_{L^{k_1}}\\
     &=& \eta_1+Ck^pt^{\beta_1-\beta_3-\beta_2p+1}+C \|w_1\|_{L^{k_1}}\\
     &=& \eta_1+Ck^p+C \|w_1\|_{L^{k_1}}
\end{eqnarray*}
and
\begin{eqnarray*}
  t^{\beta_2}\|\Theta_2(U)(t)\|_{L^{q_2}}&\leq& \eta_2+Ct^{\beta_2}\int_0^t(t-s)^{-\beta_4} \|u(s)\|^q_{L^{q_1}}\,ds+Ct^{\beta_2}\int_0^ts^\sigma(t-s)^{-\beta_6} \|w_2\|_{L^{k_2}}\,ds\\
    & =& \eta_2+Ct^{\beta_2}\int_0^t(t-s)^{-\beta_4} \|u(s)\|^q_{L^{q_1}}\,ds+C t^{\beta_2+\sigma-\beta_6+1}\|w_2\|_{L^{k_2}}\\
     &\leq& \eta_2+Ck^qt^{\beta_2}\int_0^t(t-s)^{-\beta_4} s^{-\beta_1q}\,ds+C \|w_2\|_{L^{k_2}}\\
     &=& \eta_2+Ck^q t^{\beta_2-\beta_4-\beta_1q+1}+C \|w_2\|_{L^{k_2}}\\
     &=& \eta_2 +Ck^q+C \|w_2\|_{L^{k_2}}.
\end{eqnarray*}
By choosing $k>0$ small enough and using the fact that the initial data and forcing terms are sufficiently small, we deduce that
  $$
  t^{\beta_1}\|\Phi_1(U)(t)\|_{L^{q_1}}+ t^{\beta_2}\|\Phi_2(U)(t)\|_{L^{q_2}}\leq  \eta_1+\eta_2+C \|w_1\|_{L^{k_1}}+C \|w_2\|_{L^{k_2}}+C(k^p+k^q)\leq k,
$$
which yields
$$
\Theta(\Xi)\subset \Xi.
$$
Similar calculations show that $\Theta: \Xi\to \Xi$ is a contraction, so
it has a  fixed point \break $U=(u,v)\in L^\infty\left((0,\infty),L^{q_1}(\mathbb{R}^N)\times L^{q_2}(\mathbb{R}^N)\right)$, which is a global solution to  \eqref{MILD}.

Now, we shall prove that 
\begin{equation}\label{goal}
U\in C([0,\infty),C_0(\mathbb{R}^N)\times C_0(\mathbb{R}^N)).
\end{equation}
First, let us  show that for $T>0$ small enough, one has 
$U\in C([0,T],C_0(\mathbb{R}^N)\times C_0(\mathbb{R}^N))$. Indeed, for any $T>0$ (small enough), one observes that the above argument yields uniqueness in
$$
\Xi_T=\left\{U=(u,v)\in
    L^\infty((0,T),L^{q_1}(\mathbb{R}^N)\times L^{q_2}(\mathbb{R}^N)):\;\mbox{ess}\sup_{0<t<T}\left(t^{\beta_1}\|u(t)\|_{L^{q_1}}+t^{\beta_2}\|v(t)\|_{L^{q_2}}\right)\leq k\right\}.
$$
Let $\widetilde{U}=(\tilde{u},\tilde{v})$ be the local mild solution to \eqref{1}--\eqref{Initialcondition1} constructed in Theorem \ref{local}.
Since 
$$
(u_0,v_0)\in
(C_0(\mathbb{R}^N)\times C_0(\mathbb{R}^N))\cap (L^{d_1}(\mathbb{R}^N)\times L^{d_2}(\mathbb{R}^N)),\quad q_i>d_i,\,\,  i=1,2,
$$
one deduces that  $(u_0,v_0)\in L^{q_1}(\mathbb{R}^N)\times L^{q_2}(\mathbb{R}^N)$. Hence, by assertion (iii) of Theorem \ref{local}, it holds that 
$$
\widetilde{U}\in C\left([0,T_{\max}),L^{q_1}(\mathbb{R}^N)\times L^{q_2}(\mathbb{R}^N)\right)\cap C\left([0,T_{\max}),C_0(\mathbb{R}^N)\times C_0(\mathbb{R}^N)\right).
$$
It follows that $\displaystyle\sup_{0<t<T}\left(t^{\beta_1}\|\tilde{u}(t)\|_{L^{q_1}}+t^{\beta_2}\|\tilde{v}(t)\|_{L^{q_2}}\right)\leq k$ if $T>0$ is sufficiently small. Therefore, by uniqueness, one obtains $U=\widetilde{U}$ on $[0,T]$, so that 
\begin{equation}\label{belongs1}
U\in C([0,T],C_0(\mathbb{R}^N)\times C_0(\mathbb{R}^N)).
\end{equation}
Next, using a bootstrap argument, we shall show that $U\in C([T,\infty),C_0(\mathbb{R}^N)\times C_0(\mathbb{R}^N))$. Indeed, for $t>T$, we write
\begin{eqnarray*}
u(t)-S(t)u_0-\int_0^ts^\sigma S(t-s)w_1\,ds &=&
  \int_0^TS(t-s)|v(s)|^{p}(s)\,ds+\int_T^tS(t-s)|v(s)|^{p}(s)\,ds\\
   &=:& I_1(t)+I_2(t)
\end{eqnarray*}
and
\begin{eqnarray*}
v(t)-S(t)v_0-\int_0^ts^\gamma S(t-s)w_2\,ds &=&
  \int_0^TS(t-s)|u(s)|^{q}(s)\,ds+\int_T^tS(t-s)|u(s)|^{q}(s)\,ds\\
   &=:& J_1(t)+J_2(t).
\end{eqnarray*}
On the other hand, from \eqref{belongs1}, one has $I_1,J_1\in
C([T,\infty),C_0(\mathbb{R}^N))$. Also, by the calculations used to
construct the fixed point, using the fact that $t^{-{\beta_i}}\leq
T^{-{\beta_i}}<\infty$ and  $q_i>k_i$, $i=1,2$, it follows that 
\begin{eqnarray}\label{37}
\left\{\begin{array}{lll}
I_1\in C([T,\infty),C_0(\mathbb{R}^N)\times L^{q_1}(\mathbb{R}^N)),\\ \\
J_1\in C([T,\infty),C_0(\mathbb{R}^N)\times L^{q_2}(\mathbb{R}^N))
\end{array}
\right.
\end{eqnarray}
and
\begin{eqnarray}\label{38}
\left\{\begin{array}{lll}
\displaystyle \int_0^ts^\sigma S(t-s)w_1\,ds \in C([T,\infty),C_0(\mathbb{R}^N)\cap L^{q_1}(\mathbb{R}^N)),\\ \\
\displaystyle \int_0^ts^\gamma S(t-s)w_2\,ds \in C([T,\infty),C_0(\mathbb{R}^N)\cap L^{q_2}(\mathbb{R}^N)).
\end{array}
\right.
\end{eqnarray}
Next, \eqref{31} implies that
$$
q_1<\frac{Nq_2}{Np-2q_2}\quad\mbox{and}\quad q_2<\frac{Nq_1}{Nq-2q_1}.
$$
Therefore, there exist $r_i\in(q_i,\infty]$, $i=1,2$, such that
$$
\frac{N}{2}\left(\frac{p}{q_2}-\frac{1}{r_1}\right)<1\quad\mbox{and}\quad \frac{N}{2}\left(\frac{q}{q_1}-\frac{1}{r_2}\right)<1.
$$
Then, for $\widetilde{T}>T$, using \eqref{7} and the fact that $|u|^{q}\in L^\infty((T,\widetilde{T}),L^{\frac{q_1}{q}}(\mathbb{R}^N))$ and $|v|^{p}\in L^\infty((T,\widetilde{T}),L^{\frac{q_2}{p}}(\mathbb{R}^N))$, we deduce that $I_2\in
C([T,\infty),L^{r_1}(\mathbb{R}^N))$ and $J_2\in
C([T,\infty),L^{r_2}(\mathbb{R}^N))$. By \eqref{37}--\eqref{38}, and the fact that 
$$
\left\{\begin{array}{lll}
S(t)u_0\in C([T,\infty),C_0(\mathbb{R}^N))\cap
C([T,\infty),L^{q_1}(\mathbb{R}^N)),\\\\
S(t)v_0\in C([T,\infty),C_0(\mathbb{R}^N))\cap
C([T,\infty),L^{q_2}(\mathbb{R}^N)),
\end{array}
\right.
$$
we infer that $U\in C([T,\infty),L^{r_1}(\mathbb{R}^N)\times L^{r_2}(\mathbb{R}^N))$. Iterating this procedure a finite number of times, one deduces that  
\begin{equation}\label{OKK}
U\in C([T,\infty),C_0(\mathbb{R}^N)\times C_0(\mathbb{R}^N)).
\end{equation} 
Hence, \eqref{goal} follows from \eqref{belongs1} and \eqref{OKK}. This completes the proof of Theorem \ref{global}.
\end{proof}

\vspace{1cm}

\noindent Ahmad Z. FINO\\
Department of Mathematics, Faculty of Sciences, Lebanese University, P.O. Box 826, Tripoli, Lebanon\\
E-mail: ahmad.fino01@gmail.com; afino@ul.edu.lb \\

\noindent Mohamed JLELI\\
Department of Mathematics, College of Science, King Saud University, P.O. Box 2455, Riyadh, 11451, Saudia Arabia\\
E-mail: jleli@ksu.edu.sa\\

\noindent Bessem SAMET\\
Department of Mathematics, College of Science, King Saud University, P.O. Box 2455, Riyadh, 11451, Saudia Arabia\\
E-mail: bsamet@ksu.edu.sa


\begin{thebibliography}{99}

\bibitem{AW}
D.G. Aronson,  H. Weinberger, Multidimensional nonlinear diffusion arising in population genetics, Adv. in Math. 30 (1978), 33--76.

\bibitem{BLZ}
C. Bandle, H.A. Levine, Qi S. Zhang, Critical exponents of Fujita type for inhomogeneous parabolic equations and systems, J. Math. Anal. Appl. 251 (2000), 624--648.

\bibitem{CY}
Y. Cao, J.X. Yin, Small perturbation of a semilinear pseudo-parabolic equation, Discrete Contin. Dyn. Syst. 36 (2) (2016), 631--642.


\bibitem{CH}
T. Cazenave, A. Haraux, Introduction aux probl\`emes d'\'evolution 
semi-lin\'eaires, Ellipses, Paris, 1990.




\bibitem{EH}
M. Escobedo, M.A. Herrero, Boundedness and blow up for a semilinear reaction-diffusion system, J. Differ. Equ. 89 (1991),  176--202.

\bibitem{Evans}   
L.C. Evans,  Partial Differential Equations, Vol. 19,  Graduate Studies in Mathematics, American Mathematical Society, Providence, 1998.

\bibitem{FinoKirane} 
A.Z. Fino, M. Kirane, Qualitative properties of solutions to a time-space fractional evolution equation, J. Quarterly of Applied Mathematics. 70 (2012), 133--157.


\bibitem{Fujita}
H. Fujita, On the blowup of solutions of the Cauchy problem for $u_t=\Delta u +u^{1+\alpha}$,  J. Fac. Sci., Univ. Tokyo, Sect. I. 13 (1966),  
109--124.


\bibitem{Giga} 
M. Giga, Y. Giga, J. Saal,  Nonlinear Partial Differential Equations,  Asymptotic Behavior of Solutions and Self-Similar Solutions, Birkh\"auser, Boston, 2010.



\bibitem{JKS} 
M. Jleli, T. Kawakami, B. Samet, Critical behavior for a semilinear parabolic equation with forcing term depending on time and space, J. Math. Anal. Appl. 486 (2)  (2020), 123931.



\bibitem{KS}
K.  Kobayashi, T. Sirao,  H. Tanaka,  On the growing up problem for 
semilinear heat equations, J. Math. Soc. Japan. 29 (1977), 407--429.

\bibitem{MP}  
E. Mitidieri, S. I. Pohozaev,  A priori estimates and
blow-up of solutions to nonlinear partial differential equations and
inequalities, Proc. Steklov. Inst. Math.  234 (2001), 1--383.

\bibitem{Y}
J. Yang, S. Zheng, C. Qu, Fujita phenomenon in inhomogeneous fast diffusion system, Z. Angew. Math. Phys. 64 (2013), 311--319.



\bibitem{Zeng}
X. Zeng, The critical exponents for the quasi-linear parabolic equations with inhomogeneous terms, J. Math. Anal. Appl. 332 (2007), 1408--1424.


\bibitem{ZhangV}
Qi S. Zhang, Blow up and global existence of solutions to an inhomogeneous parabolic system, J. Differential Equations. 147 (1998), 155--183.


\bibitem{Zhang} 
Qi S. Zhang,  A blow up result for a nonlinear wave equation with damping:
the critical case, C. R. Acad. Sci. Paris. 333 (2001), 109--114.
\end{thebibliography}
\end{document}